\documentclass[11pt]{amsart}
\usepackage[margin=1in]{geometry}
\usepackage{amsmath}
\usepackage{amsfonts,amsmath}
\usepackage{paralist}
\usepackage[colorlinks=true]{hyperref}
\hypersetup{urlcolor=blue, citecolor=red}
\numberwithin{equation}{section}

\newcommand{\po}{\partial\Omega}

\newcommand{\io}{\int_{\Omega}}
\newcommand{\ioT}{\int_{\Omega_{T}}}

\newcommand{\erb}{\eta(f;r;B_r(x_0)) }

\newcommand{\iot}{\int_{\Omega_{\tau}}}

\newcommand{\ot}{\Omega_T }

\newcommand{\ep}{\varepsilon}

\newcommand{\mdiv}{\textup{div}}
\newcommand{\q}{\mathbf{q}}
\newcommand{\qe}{\mathbf{q}^{(\varepsilon)}}
\newcommand{\uxo}{u_{x_1}}
\newcommand{\vxt}{v_{x_2}}
\newcommand{\vxo}{v_{x_1}}

\newcommand{\vp}{\varphi}

\newcommand{\mdet}{\textup{det}}

\newtheorem{theorem}{Theorem}[section]

\newtheorem{lemma}[theorem]{Lemma}

\theoremstyle{definition}

\title[a groundwater flow problem
] %Use the shortened version of the full title
      {Global existence of strong solutions to a groundwater flow problem}

\author[Xiangsheng Xu]{}

 \email{xxu@math.msstate.edu}
\begin{document}
\maketitle

% Enter the first author's name and address:
\centerline{\scshape Xiangsheng Xu}
\medskip
{\footnotesize
% please put the address of the first author
 \centerline{Department of Mathematics \& Statistics}
   \centerline{Mississippi State University}
   \centerline{ Mississippi State, MS 39762, USA}
} % Do not forget to end the {\footnotesize by the sign }

\bigskip

%The abstract of your paper% in two space dimensions, 
%We show that the second equation in the system is actually uniformly parabolic, i.e., $|\nabla v|\in L^\infty$. Furthermore, we also have  $|\nabla u|\in L^\infty$. 
% and then suitably applying the De Giorge iteration method to the equation.
\begin{abstract}
In this paper we study the initial boundary value problem for the system $\Delta v= u_{x_1},\ u_t-\mbox{div}\left(\left((a|\mathbf{q}|+m)I+(b-a)\frac{\mathbf{q}\otimes\mathbf{q}}{|\mathbf{q}|}\right)\nabla u\right)=-\nabla u\cdot\mathbf{q}$,
where $\mathbf{q}=(-v_{x_2}, v_{x_1})^T$, $\mathbf{q}\otimes\mathbf{q}=\mathbf{q}\mathbf{q}^T$. This problem has been proposed as a model for a fluid flowing through a porous medium under the influence of gravity and hydrodynamic dispersion. For each $T>0$ we obtain a so-called strong solution $(v, u)$ in the function space $L^\infty(0,T; \left(W^{1,\infty}(\Omega)\right)^2)$, where $\Omega$ is a bounded domain in $\mathbb{R}^2$.
The key ingredient in our approach is the decomposition $A^2=\mbox{tr} (A)A-\mbox{det}(A) I$ for any $2\times 2$ symmetric matrix $A$.  By exploring this decomposition, we are able to derive an equation of parabolic type for the function $\left(\left((a|\mathbf{q}|+m)I+(b-a)\frac{\mathbf{q}\otimes\mathbf{q}}{|\mathbf{q}|}\right)\nabla u\cdot\nabla u\right)^j, j\geq 1$. With the aid of this equation we obtain a uniform bound for $\nabla u$. {\it Z. angew. Math. Phys.}, to appear.

\end{abstract}
%\mathbb{R}^2$due to Darcy's law  $\mathbb{R}^2$, also known as trilinear forms, obtain a weak solution $(u,v) $ with 
%The title of your section 1%The mathematical challenge is due to the presence of the so-called cubic nonlinearities in the system. 
\section{Introduction}

Let $\Omega$ be a bounded domain in the $x=(x_1,x_2)$ plane with boundary $\po$ and $T$ any positive number. We study the problem
\begin{eqnarray}
\Delta v&=& u_{x_1}\ \ \mbox{in $\ot\equiv\Omega\times(0,T)$},\label{gwe1}\\
u_t-\mdiv\left(\left((a|\mathbf{q}|+m)I+(b-a)\frac{\q\otimes\q}{|\q|}\right)\nabla u\right)&=& -\nabla u\cdot\q\ \ \ \mbox{in $\ot$},\label{gwe2}\\
\left((a|\mathbf{q}|+m)I+(b-a)\frac{\q\otimes\q}{|\q|}\right)\nabla u\cdot\nu&=&0\ \ \ \mbox{on $\Sigma_T\equiv\po\times(0,T)$},\label{gwe3}\\
v&=& 0\ \ \ \mbox{on $\Sigma_T$},\label{gqe4}\\
u(x,0)&=&u_0(x)\ \ \mbox{on $\Omega$},\label{gwe5}
\end{eqnarray}
where
\begin{equation}\label{qd}
\q=\left(\begin{array}{c}
-v_{x_2}\\
v_{x_1}
\end{array}\right),
\end{equation}
%is the so-called stream function, 
$\q\otimes\q=\q\q^T$, $I$ is the $2\times 2$ identity matrix, $a, b, m$ are positive numbers with $b>a$, and $\nu$ is the unit outward normal to the $\po$. 

This system arises in the description of the movement of a fluid of variable density
$u$ through a porous medium under the influence of gravity and hydrodynamic
dispersion \cite{CDL}. The first equation \eqref{gwe1} is derived from Darcy’s law, while the second equation \eqref{gwe2} describes the mass balance. See \cite{CDL,PS} for details.
In a slightly different form, problem \eqref{gwe1}-\eqref{gwe5} was studied by Su \cite{S} using classical
partial differential equation (PDE) methods. In \cite{CDL}  the problem was formulated as abstract evolution equations in Banach spaces. To describe the results there, we set
% Two cases were considered. First, if the coefficient matrix in \cite{CDL}
\begin{equation}\label{ddf}
D\equiv (a|\mathbf{q}|+m)I+(b-a)\frac{\q\otimes\q}{|\q|}
\end{equation}
If $D$ can be taken as a constant multiple of the identity matrix $I$, the resulting problem has a classical solution, while in the general case, 
% As a result, we obtain 
only local existence of weak solutions in
$W^{1,p}(\Omega)$ for some $p> 1$ was obtained. The global existence was thereby left as an open problem.
%In particular,  the fact that
%the coefficient matrix $D$  is not differentiable at the origin was mentioned as an impediment to the existence of classical solutions. 
%This turns out to be quite efficient because of the particular form of the problem.

The objective of this paper is to solve the problem left open in \cite{CDL}. Before we precisely state our result, we make some preliminary observations. 
By the definition of $\q$, we always have
\begin{equation}\label{diq}
\mdiv\q=0.
\end{equation}
Thus 
\begin{equation*}
\nabla u\cdot\q=\mdiv(u\q).
\end{equation*}
Moreover,
\begin{equation*}
\left|\frac{\q\otimes\q}{|\q|}\right|\leq |\q|.
\end{equation*}
It is natural for us to define 
\begin{equation*}
\frac{\q\otimes\q}{|\q|}=0\ \ \mbox{whenever $\q=0$.}
\end{equation*}
Thus the coefficient matrix $D$ 
%given by
%\begin{equation}\label{ddf}
%D=(a|\mathbf{q}|+m)I+(b-a)\frac{\q\otimes\q}{|\q|}
%\end{equation}
is well-defined and satisfies
\begin{equation}\label{ellip}
(a|\mathbf{q}|+m)|\xi|^2\leq D\xi\cdot\xi=(a|\mathbf{q}|+m)|\xi|^2+\frac{b-a}{|\q|}(\q\cdot\xi)^2\leq (b|\mathbf{q}|+m)|\xi|^2\ \ \mbox{for each $\xi\in \mathbb{R}^2$}.
\end{equation}
%The second equation in our system becomes singular on the set where $|\q|$ is infinity. 
As observed in \cite{CDL}, each entry of $D$ is a Lipschitz function of $\q$. In particular, we have
\begin{eqnarray}
|D_t|&\leq & c|\nabla v_t|,\label{dpt}\\
|D_{x_1}|+|D_{x_2}|&\leq& c|\nabla^2 v|,\label{dpx}
\end{eqnarray}
where $\nabla^2 v$ denotes the Hessian of $v$. The letter $c$ here and in what follows  represents a generic positive number whose value can be derived from the given data at least in theory. Our main result is:
\begin{theorem}[Main Theorem]Let $a, b, m$ be given as before. Assume:
\begin{enumerate}
	\item[\textup{(H1)}]$\Omega$ is a bounded domain in $\mathbb{R}^2$ with $C^{2,\gamma}$ boundary $\partial\Omega$
	 for some $\gamma\in(0,1)$;
	 % each $p>1$
	\item[\textup{(H2)}] $u_0\in C^2(\overline{\Omega})$.
\end{enumerate}
 Then for each $T>0$ there is a weak solution $(v, u)$ to \eqref{gwe1}-\eqref{gwe5} with 
 \begin{equation}\label{r33}
 |\nabla u|\in L^\infty(\ot).
 \end{equation}
\end{theorem}

Assumption (H1) is more than enough for
%assumed so that we can apply 
the classical Calder\'{o}n-Zygmund estimate to hold for $v$ \cite{CFL}. It is also sufficient for the boundary estimate of $u$ (see the Appendix below). Of course, 
%the fact that $|\nabla u|, |\nabla v|\in L^\infty(\ot)$ can yield
we can deduce from \eqref{r33} more regularity results for the weak solution. In fact, one can infer from Section 3 below that equations \eqref{gwe1} and \eqref{gwe2} are satisfied in the a.e. sense. This is why we call our solution a strong solution. 

To view our main theorem in the context of the regularity theory for parabolic partial differential equations, we write \eqref{gwe2} in the form
\begin{equation*}
u_t-\mdiv\left(D\nabla u\right)=-\mdiv\left(u\q\right)\ \ \ \mbox{in $\ot$}.
\end{equation*}
If $D$ is continuous on $\overline{\ot}$, we can deduce from a result in (\cite{BLP}, p.273-274)  that for each $p>2$ there is a positive number $c$ such that
\begin{equation}\label{r22}
\|\nabla u\|_{p,\ot}\leq c\| u\q\|_{p,\ot}+ c\|\nabla u\|_{2,\ot}+c\| u_0\|_{W^{1, p}(\Omega)}.
\end{equation}
(Also see the proof of Theorem 8 in \cite{X2}). This is the parabolic version of the classical $W^{1,p}$ estimate (\cite{R}, p.82). In general, the above inequality fails when $p=\infty$. 
%However, i
In our case 
%estimate \eqref{r22} can lead to the conclusion of our main theorem. We shall indicate this near the end of Section 3. Obviously, 
the continuity of $D$ is an easy consequence of
\begin{equation}\label{r30}
\q\in  \left(C(\overline{\ot})\right)^2,
\end{equation}
% the continuity of $\q$, 
which, in turn, relies on the regularity properties of $u$. As we shall see near the end of Section 3, 
%the membership of $\q$ in $ \left(C(\overline{\ot})\right)^2$
\eqref{r30} is essentially equivalent to \eqref{r33}. Thus the question is: Which one of these two is easier to obtain? To explore this issue, we seek various a priori estimates.
%At first glance, \eqref{r30} seems to be
%Unfortunately, trying to derive \eqref{r30}
% from a traditional approach leads us to a dead end. To be more specific,
% It is natural for us to seek the continuity of $\q$ first. do not seem to be enough to guarantee
In the beginning of Section 3, we quickly collect the following inequalities :
\begin{equation}\label{fr}
\|u\|_{\infty,\ot}\leq c,\ \sup_{0\leq t\leq T}\|\nabla v\|_{r,\Omega}\leq c(r)\ \ \mbox{for each $r>1$},\|\nabla u\|_{2,\ot}\leq c,\ \mbox{and}\ \| v\|_{L^2(0,T; W^{2,2}(\Omega))}\leq c.
\end{equation}
Even though our problem bears some resemblance to the Patlak-Keller-Segel model in chemotaxis \cite{HP}, 
%in light of \eqref{diq} and \eqref{ddf} 
our system is not a cross-diffusion one due to \eqref{diq} and \eqref{ddf}. Thus no blow-up in solutions has occurred.
% We will omit the details here.
In fact, since
%Furthermore,  by virtue of the fact that 
the ratio of the largest eigenvalue to the smallest one of $D$ lies in $[1, \frac{b}{a}]$, one may even wonder if it is possible to obtain the H\"{o}lder continuity of $u$ \cite{PS2}. Nonetheless, these results are not enough for our purpose.
% we still do not see the light at the end of the tunnel.
%these results seem to offer little help.
 Another possibility is to pursue a parabolic version of the result in \cite{M}, i.e., \eqref{r22} holds for some $p>2$. Unfortunately, even if this is true,  we still cannot reach \eqref{r33} or \eqref{r30} because $p$ can be very close to $2$ and
%. The latter is due to the fact that 
equation \eqref{gwe1} is degenerate in the time variable.
%Evidently, there is still 
%Now we realize that 
Indeed, there is a big gap between possible estimates from the classical regularity theory and our desired result.
%the continuity of $\q$ (or equivalently $\nabla v$) in $\overline{\ot}$. 
%To circumvent this, 
%we will take a totally new  approach. That is, we attempt to 
%Therefore, we must 

We will show that we can prove \eqref{r33} directly without going through \eqref{r30}. This is achieved  by deriving an equation for the function
\begin{equation}\label{r111}
\psi=\left(D\nabla u\cdot\nabla u\right)^j, \ \ j\geq 1.
\end{equation}
To be specific, we establish that for each $j\geq 1$ the function $\psi$ defined above  satisfies
\begin{equation}\label{wine1111}
\frac{1}{\psi}\psi_t-	\textup{div}\left(\frac{1}{\psi}D\nabla\psi\right)= \frac{1}{\psi}\mathbf{H}\cdot\nabla\psi+jh+j\textup{div}\mathbf{F}\ \ \mbox{in $\{|\nabla u|>0\}$}.
%\ot
\end{equation}
Here the coefficients $\mathbf{H}, h, \mathbf{F}$ are  bounded by 
%$$, 
the entries of $D$ and the partial derivatives $u_t, D_t, D_{x_1}, D_{x_2}$. It turns out that each of those partial derivatives can also be bounded by $\nabla u$. The introduction of $j$ enables us to show that \eqref{r30} implies \eqref{r33} (see Section 3 below).
%The key to our success  is that we can choose $j$ to be sufficiently large. This creates a situation where a term with large power is bounded by the same term with small power, from which estimates follow.

%The related idea of deriving an equation for a power of $|\nabla u|$
%$\ln \left(D\nabla u\cdot\nabla u\right)$ 
%goes back to \cite{B,PS1,A,S1,X7}, even though the motivations there may vary.
An equation similar to ours in the context of equations of elliptic type was derived in \cite{X7}, based upon contributions in
%The idea of obtaining an equation for a function like $\psi$ first appeared  , which was motivated by
 earlier works \cite{B,PS1,A,S1}. In our case how to proceed with the derivation is not immediately clear. It turns out that we can rely on
% can be done via 
 the following observation: Let $A$ be an $2\times 2$ symmetric matrix. Then we have
\begin{equation}\label{fora}
A^2=\mbox{tr}(A)A-\mdet(A)I.
\end{equation} 
The proof of this formula is very simple. Indeed, denote by $a_{ij}$ the $ij$ entry of $A$. We calculate
\begin{eqnarray}
A^2&=&\left(\begin{array}{cc}
a_{11}&a_{12}\\
a_{12}&a_{22}
\end{array}\right)\left(\begin{array}{cc}
a_{11}&a_{12}\\
a_{12}&a_{22}
\end{array}\right)\nonumber\\
&=&\left(\begin{array}{cc}
a_{11}^2+a_{12}^2&a_{11}a_{12}+a_{12}a_{22}\\
a_{11}a_{12}+a_{12}a_{22}&a_{12}^2+a_{22}^2
\end{array}\right)\nonumber\\
&=&\left(\begin{array}{cc}
a_{11}^2+a_{11}a_{22}-\mdet(A)&a_{12}(a_{11}+a_{22})\\
a_{12}(a_{11}+a_{22})&a_{11}a_{22}-\mdet(A)+a_{22}^2
\end{array}\right)\nonumber\\
&=&(a_{11}+a_{22})\left(\begin{array}{cc}
a_{11}&a_{12}\\
a_{12}&a_{22}
\end{array}\right)-\mdet(A)I.
\end{eqnarray}
For all practical purposes the last term in \eqref{fora} behaves like a scalar function. The formula \eqref{fora} says that a quadratic function of a matrix can be decomposed into a linear function of the matrix plus a roughly scalar function. That is, the first term in the decomposition has weaken the non-linearity, while the second term has reduced the dimensionality.  In this work, we do not actually use the formula directly. What we use is the idea behind the proof of this formula. That is, whenever we can represent certain terms in a high-powered matrix in terms of a determinant, something ``good'' follows. 
It is this tenet that 
%enables This 
guides us to an equation of parabolic type for $\psi$.

%on the left-hand side. 

We believe that decomposition such as \eqref{fora} is a very powerful tool. %and it should be useful elsewhere. 
% One such possibility is that
One possibility is that one can explore the relationship between the cubic power of a three-by-three matrix and its determinant for potential applications.

This work is organized as follows: Section 2 is largely devoted to the derivation of \eqref{wine1111}. In section 3, we assume that problem \eqref{gwe1}-\eqref{gwe5} has a regular enough solution and proceed to derive a priori estimates for the solution. The main theorem is established as a consequence of these estimates. Near the end of the section, we construct a sequence of smooth approximate solutions, thereby justifying the regularity assumptions made earlier in the calculations. In the Appendix, we provide some additional details about the boundary estimates involved.

\section{Derivation of Equation \eqref{wine1111}}

In this section we derive \eqref{wine1111}. It is essentially the parabolic version of the result in \cite{X7}. But before we do that, we recall some definitions and known results and formulae.

If $A(x)$ is a matrix-valued function then
\begin{eqnarray*}
\mbox{div}A(x)&=&\mbox{the row vector whose $i$-th entry is the divergence of the  $i$-th column of $A$ } \nonumber\\
&=&(\mbox{div}\mathbf{A}_1,\mbox{div}\mathbf{A}_2).
\end{eqnarray*}
When $\mathbf{G(x)}$ is a vector-valued function, then
\begin{equation*}
\nabla\mathbf{G(x)}=\mbox{the $2\times 2$ matrix whose $ij$-entry is $(g_j(x))_{x_i}$}
 =(\nabla g_1, \nabla g_2).
\end{equation*}
%Denote by $\nabla^2u$  the Hessian of $u$. Then we have
The following identities will be frequently used
\begin{eqnarray}
\nabla \left(\mathbf{F}\cdot\mathbf{G}\right)&=&\nabla \mathbf{F}\mathbf{G}+\nabla\mathbf{G}\mathbf{F},\label{form1}\\
\mbox{div}\left(A\mathbf{F}\right)&=& A:\nabla\mathbf{F} +\mbox{div}A\mathbf{F},\label{form2}\\
\nabla\left(A\mathbf{F}\right)&=& \nabla\mathbf{F}A^T+\left(
A_{x_1}\mathbf{F},
A_{x_2}\mathbf{F}
\right)^T,\label{form3}\\
\mbox{div}(uA)&=&u\mbox{div}A+(\nabla u)^TA,\label{form4}\\
\nabla|\nabla u|^2&=& 2\nabla^2u\nabla u.\nonumber
\end{eqnarray}
%We also need the interpolation inequality
%\begin{equation}\label{inter}
%\|u\|_q\leq \varepsilon\|u\|_r+\varepsilon^{-\mu}\|u\|_\ell,
%\end{equation}
%where $1\leq \ell\leq q\leq r$ with $\mu=\left(\frac{1}{\ell}-\frac{1}{q}\right)/\left(\frac{1}{q}-\frac{1}{r}\right)$.

The next lemma deals with sequences of non-negative numbers
which satisfy certain recursive inequalities.
\begin{lemma}\label{ynb}
	Let $\{y_n\}, n=0,1,2,\cdots$, be a sequence of positive numbers satisfying the recursive inequalities
	\begin{equation*}
	y_{n+1}\leq cb^ny_n^{1+\alpha}\ \ \mbox{for some $b>1, c, \alpha\in (0,\infty)$.}
	\end{equation*}
	If
	\begin{equation*}
	y_0\leq c^{-\frac{1}{\alpha}}b^{-\frac{1}{\alpha^2}},
	\end{equation*}
	then $\lim_{n\rightarrow\infty}y_n=0$.
\end{lemma}
This lemma can be found in (\cite{D}, p.12).

Define
\begin{equation}
\erb=\sup_{x\in \mathbb{R}^2}\int_{B_{ r}(x_0)}|f(y)|\chi_{B_{ r}(x_0)}|\ln|x-y||dy.
\end{equation}
Here $B_{ r}(x_0)$ denotes the open ball with center $x_0$ and radius $r$ and  $\chi_{B_{ r}(x_0)}$ is the indicator function of the set $B_{ r}(x_0)$.
We say $f\in K_2(B_{ r}(x_0))$, the class of Kato functions, if $\lim_{r\rightarrow 0}\erb=0$ \cite{K}. Suppose that $f\in L^p(B_{ r}(x_0))$ for some $p>1$. Then we have
\begin{equation}\label{kcl}
\erb\leq \|f\|_{p,B_{ r}(x_0)}\left(\int_{B_{ r}(x_0)}|\ln|x-y||^{\frac{p}{p-1}}dy\right)^{\frac{p=1}{p}}\leq c(\varepsilon)r^{2-\varepsilon} \|f\|_{p,B_{ r}(x_0)},\ \ \ep\in(0,1).
\end{equation}
For each $f\in K_2(B_{ r}(x_0))$, there is a positive number $c$ such that
\begin{equation}\label{kin}
\int_{B_{ r}(x_0)}|f|w^2dx\leq c\erb\left(\int_{B_{ r}(x_0)}|\nabla w|^2dx+\frac{1}{r^2}\int_{B_{ r}(x_0)}w^2dx\right)
\end{equation}
for each $w\in W^{1,2}(B_{ r}(x_0))$.
This inequality is essentially the two-dimensional version of Lemma 1.1 in \cite{FGL}. Also see \cite{K,X9}.

We write \eqref{gwe2} in the form
\begin{equation}\label{happy1}
u_t-D:\nabla^2 u=\mdiv D\nabla u-\nabla u\cdot\q\equiv w.
\end{equation}
Denote by $d_{ij}$ the entry of $D$ that lies in the $i^{\mbox{th}}$ row and the $j^{\mbox{th}}$ column. Then we have
\begin{equation}\label{happy2}
u_t-(d_{11}u_{x_1x_1}+2d_{12}u_{x_1x_2}+d_{22}u_{x_2x_2})= w.
\end{equation}
We introduce the following quantities:
\begin{eqnarray}
\vp&=& D\nabla u\cdot\nabla u=d_{11}u_{x_1}^2+2d_{12}u_{x_1}u_{x_2}+d_{22}u_{x_2}^2,\label{adef}\\
D_1&=&\left(\begin{array}{cc}
d_{11}(d_{11}u_{x_1}+d_{12}u_{x_2})&d_{12}d_{11}u_{x_1}-(d_{22}d_{11}-2d_{12}^2)u_{x_2}\\
d_{11}(d_{12}u_{x_1}+d_{22}u_{x_2})&d_{22}(d_{11}u_{x_1}+d_{12}u_{x_2})
\end{array}\right),\label{dodef}\\
D_2&=&\left(\begin{array}{cc}
d_{11}(d_{12}u_{x_1}+d_{22}u_{x_2})&d_{22}(d_{11}u_{x_1}+d_{12}u_{x_2})\\
-(d_{22}d_{11}-2d_{12}^2)u_{x_1}+d_{12}d_{22}u_{x_2}&d_{22}(d_{12}u_{x_1}+d_{22}u_{x_2})
\end{array}\right),\label{dtdef}\\
D_3&=&-D\nabla u (D\nabla u)^T
%\left(\begin{array}{cc}
%(d_{11}u_{x_1}+d_{12}u_{x_2})^2&(d_{11}u_{x_1}+d_{12}u_{x_2})(d_{12}u_{x_1}+d_{22}u_{x_2})\\
%(d_{11}u_{x_1}+d_{12}u_{x_2})(d_{12}u_{x_1}+d_{22}u_{x_2})&(d_{12}u_{x_1}+d_{22}u_{x_2})^2
%\end{array}\right)
,\label{dthdef}\\
\mathbf{G}&=&\vp^{-1}\left(\begin{array}{c}
D_{x_1}\nabla u\cdot\nabla u\\
D_{x_2}\nabla u\cdot\nabla u
\end{array}\right).\label{gdef}
\end{eqnarray}

\begin{theorem}\label{keyth} For each $j\geq 1$ the function $\psi=\vp^j$ satisfies the equation 
	%\eqref{wine111}
	\begin{equation}\label{wine111}
\frac{1}{\psi}\psi_t-	\textup{div}\left(\frac{1}{\psi}D\nabla\psi\right)= \frac{1}{\psi}\mathbf{H}\cdot\nabla\psi+jh+j\textup{div}\mathbf{F}\ \ \textup{in $\{|\nabla u|>0\}$},
	\end{equation}
	where
	\begin{eqnarray}
	\mathbf{H}&=&D\mathbf{G}+2\vp^{-1}u_tD\nabla u+\frac{1}{\textup{det}(D)\vp}\left(
	D_1^T\nabla\mdet(D),
	D_2^T\nabla\mdet(D)
	\right)\nabla u,\label{hdef}\\
	\mathbf{F}&=&-D\mathbf{G}+2\vp^{-1}wD\nabla u,\label{kdef}\\
	h&=&-\frac{1}{\textup{det}(D)\vp}\nabla\mdet(D)\cdot\left(D_1\mathbf{G}, D_2\mathbf{G}\right)\nabla u+\frac{2(u_t-w)}{\textup{det}(D)\vp^2}\nabla\mdet(D)\cdot D_3\nabla u\nonumber\\
	&&-2\vp^{-1}u_t(u_t+\mdiv D\nabla u-w)+\vp^{-1}D_t\nabla u\cdot\nabla u\nonumber\\
	&&-2\vp^{-1}(u_t-w)D\nabla u\cdot \mathbf{G}-D\mathbf{G}\cdot \mathbf{G}.\label{shdef}
	\end{eqnarray}
\end{theorem}
\begin{proof}
 This theorem is the parabolic version of a result in \cite{X7}. As is done there, we first derive an equation for
\begin{equation*}
\phi=\ln \vp.
\end{equation*}
%To gain some insights into this theorem, here we offer a proof in the case where 
We calculate
\begin{eqnarray}
\phi_t-\mdiv\left(D\nabla \phi\right)&= &\frac{1}{\vp}\vp_t-\mdiv\left(\frac{1}{\vp}D\nabla \vp\right)\nonumber\\
&=& \frac{1}{\vp}\vp_t-\frac{1}{\vp}\mdiv\left(D\nabla \vp\right)+\frac{1}{\vp^2}D\nabla\vp\cdot\nabla\vp\nonumber\\
&=&\frac{1}{\vp}\left(\vp_t-\mdiv\left(D\nabla\vp\right)+\frac{1}{\vp}D\nabla\vp\cdot\nabla\vp\right).\label{happy7}
\end{eqnarray}
Take the gradient of \eqref{happy1} and take the dot-product of the resulting equation with $D\nabla u$ to obtain
\begin{equation*}
D\nabla u\cdot\nabla u_t-D\nabla u\cdot\nabla\left(D:\nabla^2 u\right)=D\nabla u\cdot\nabla w.
\end{equation*}
Subsequently,
%It is easy to see that
\begin{eqnarray*}
\vp_t&=&2D\nabla u\cdot\nabla u_t+D_t\nabla u\cdot\nabla u\nonumber\\
&=&2D\nabla u\cdot\nabla\left(D:\nabla^2 u\right)+2D\nabla u\cdot\nabla w+D_t\nabla u\cdot\nabla u\nonumber\\
&=&2D\nabla u\cdot\nabla\left(D:\nabla^2 u\right)+2\mdiv\left(wD\nabla u\right)-2w\mdiv\left(D\nabla u\right)+D_t\nabla u\cdot\nabla u.
\end{eqnarray*}
We evaluate
\begin{eqnarray*}
D\nabla u\cdot\nabla\left(\nabla^2u:D\right)
&=&\left(d_{11}u_{x_1x_1}+2d_{12}u_{x_1x_2}+d_{22}u_{x_2x_2}\right)_{x_1}\left(d_{11}u_{x_1}+d_{12}u_{x_2}\right)\nonumber\\
&&+\left(d_{11}u_{x_1x_1}+2d_{12}u_{x_1x_2}+d_{22}u_{x_2x_2}\right)_{x_2}\left(d_{12}u_{x_1}+d_{22}u_{x_2}\right)\nonumber\\
%&&+D\nabla u\cdot\nabla\left(\mdiv D\nabla u\right)\nonumber\\
&=& \mdiv\left(\begin{array}{c}
\left(d_{11}u_{x_1x_1}+2d_{12}u_{x_1x_2}+d_{22}u_{x_2x_2}\right)\left(d_{11}u_{x_1}+d_{12}u_{x_2}\right)\\
\left(d_{11}u_{x_1x_1}+2d_{12}u_{x_1x_2}+d_{22}u_{x_2x_2}\right)\left(d_{12}u_{x_1}+d_{22}u_{x_2}\right)
\end{array}\right)\nonumber\\
&&-(\nabla^2u:D)\mdiv\left(D\nabla u\right)\nonumber\\
%&&+D\nabla u\cdot\nabla\left(\mdiv D\nabla u\right)\nonumber\\
&=& \mdiv\left(\begin{array}{c}
\left(d_{11}u_{x_1x_1}+2d_{12}u_{x_1x_2}+d_{22}u_{x_2x_2}\right)\left(d_{11}u_{x_1}+d_{12}u_{x_2}\right)\\
\left(d_{11}u_{x_1x_1}+2d_{12}u_{x_1x_2}+d_{22}u_{x_2x_2}\right)\left(d_{12}u_{x_1}+d_{22}u_{x_2}\right)
\end{array}\right)\nonumber\\
&&-(u_t-w)\mdiv\left(D\nabla u\right).
\end{eqnarray*}
Then calculate from \eqref{happy2} and \eqref{gdef} that
\begin{eqnarray*}
D\nabla \vp&=& D\left(\begin{array}{c}
2d_{11}u_{x_1}u_{x_1x_1}+2d_{12}(u_{x_2}u_{x_1x_1}+u_{x_1}u_{x_1x_2})+2d_{22}u_{x_2}u_{x_2x_1}\\
2d_{11}u_{x_1}u_{x_1x_2}+2d_{12}(u_{x_2}u_{x_1x_2}+u_{x_1}u_{x_2x_2})+2d_{22}u_{x_2}u_{x_2x_2}
\end{array}\right)\nonumber\\
&&+\vp D\mathbf{G}\nonumber\\
&=&2 D\left(\begin{array}{c}
u_{x_1x_1}(d_{11}u_{x_1}+d_{12}u_{x_2})+u_{x_2x_1}(d_{12}u_{x_1}+d_{22}u_{x_2})\\
u_{x_1x_2}(d_{11}u_{x_1}+d_{12}u_{x_2})+u_{x_2x_2}(d_{12}u_{x_1}+d_{22}u_{x_2})
\end{array}\right)+\vp D\mathbf{G}\nonumber\\
&=&2 \left(\begin{array}{c}
(d_{11}u_{x_1x_1}+d_{12}u_{x_1x_2})(d_{11}u_{x_1}+d_{12}u_{x_2})\\
(d_{12}u_{x_1x_1}+d_{22}u_{x_1x_2})(d_{11}u_{x_1}+d_{12}u_{x_2})
\end{array}\right)\nonumber\\
&&+2\left(\begin{array}{c}
(d_{11}u_{x_1x_2}+d_{12}u_{x_2x_2})(d_{12}u_{x_1}+d_{22}u_{x_2})\\
(d_{12}u_{x_1x_2}+d_{22}u_{x_2x_2})(d_{12}u_{x_1}+d_{22}u_{x_2})
\end{array}\right)+\vp D\mathbf{G}.
\end{eqnarray*}
As we mentioned in the introduction, we try to represent the difference between $2D\nabla u\cdot\nabla\left(D:\nabla^2 u\right)$ and $\mdiv\left(D\nabla \vp \right)$ in terms of determinants. To this end, we compute
\begin{eqnarray}
\lefteqn{2D\nabla u\cdot\nabla\left(D:\nabla^2 u\right)-\mdiv\left(D\nabla \vp \right)}\nonumber\\
&=&2\mdiv\left(\begin{array}{c}
\left(d_{12}u_{x_1x_2}+d_{22}u_{x_2x_2}\right)\left(d_{11}u_{x_1}+d_{12}u_{x_2}\right)-(d_{11}u_{x_1x_2}+d_{12}u_{x_2x_2})(d_{12}u_{x_1}+d_{22}u_{x_2})\\
\left(d_{11}u_{x_1x_1}+d_{12}u_{x_1x_2}\right)\left(d_{12}u_{x_1}+d_{22}u_{x_2}\right)-(d_{12}u_{x_1x_1}+d_{22}u_{x_1x_2})(d_{11}u_{x_1}+d_{12}u_{x_2})
\end{array}\right)\nonumber\\
&&-2(u_t-w)\mdiv\left(D\nabla u\right)-\mdiv\left(\vp D\mathbf{G}\right)\nonumber\\
&=&2\mdiv\left(\begin{array}{c}
\mdet(D)(u_{x_2x_2}u_{x_1}-u_{x_1x_2}u_{x_2})\\
\mdet(D)(u_{x_1x_1}u_{x_2}-u_{x_1x_2}u_{x_1})
\end{array}\right)-2(\nabla^2u:D)\mdiv\left(D\nabla u\right)-\mdiv\left(\vp D\mathbf{G}\right)\nonumber\\
&=&2\nabla\mdet(D)\cdot\left(\begin{array}{cc}
u_{x_2x_2}&-u_{x_1x_2}\\
-u_{x_1x_2}&u_{x_1x_1}
\end{array}\right)\nabla u+4\mdet(D)\mdet(\nabla^2u)\nonumber\\
&&-2(u_t-w)\mdiv\left(D\nabla u\right)-\mdiv\left(\vp D\mathbf{G}\right).\label{happy6}
\end{eqnarray}
Here we have used the fact that
\begin{equation*}
\mdiv\left(\begin{array}{c}
u_{x_2x_2}u_{x_1}-u_{x_1x_2}u_{x_2}\\
u_{x_1x_1}u_{x_2}-u_{x_1x_2}u_{x_1}
\end{array}\right)=2\mdet(\nabla^2 u).
\end{equation*}
On the other hand, we have
\begin{eqnarray}
\nabla \vp &=& \nabla^2uD\nabla u+\nabla(D\nabla u)\nabla u\nonumber\\
&=&2\nabla^2uD\nabla u+\vp \mathbf{G}.\label{happy5}
\end{eqnarray}
Consequently,
\begin{eqnarray}
D\nabla \vp \cdot\nabla \vp &=&(\nabla \vp )^TD\nabla \vp \nonumber\\
&=&(2(D\nabla u)^T\nabla^2 u+\vp \mathbf{G}^T) D\left(2\nabla^2uD\nabla u+\vp \mathbf{G}\right)\nonumber\\
&=&4(D\nabla u)^T\nabla^2 uD\nabla^2uD\nabla u+4\vp D\nabla^2uD\nabla u\cdot \mathbf{G}+\vp ^2D\mathbf{G}\cdot \mathbf{G}.\label{happy3}
%\nonumber\\
%&=&D\nabla^2uD\nabla u+aD\mathbf{G},
\end{eqnarray}
Set \begin{equation*}
B=\nabla^2 uD\nabla^2u.
\end{equation*}
 We also represent the four entries of $B$ in terms of determinants as follows
\begin{eqnarray*}
b_{11}&=&
d_{11}u_{x_1x_1}^2+2d_{12}u_{x_1x_1}u_{x_1x_2}+d_{22}u_{x_1x_2}^2\nonumber\\
&=&d_{11}u_{x_1x_1}^2+2d_{12}u_{x_1x_1}u_{x_1x_2}+d_{22}(u_{x_1x_1}u_{x_2x_2}-\mdet(\nabla^2u))\nonumber\\
&=&u_{x_1x_1}(d_{11}u_{x_1x_1}+2d_{12}u_{x_1x_2}+d_{22}u_{x_2x_2})-d_{22}\mdet(\nabla^2u)\nonumber\\
&=&u_{x_1x_1}(u_t-w)-d_{22}\mdet(\nabla^2u),\\
b_{21}&=& d_{11}u_{x_1x_1}u_{x_1x_2}+d_{12}u_{x_1x_1}u_{x_2x_2}+d_{12}u_{x_1x_2}^2+d_{22}u_{x_2x_2}u_{x_1x_2}\nonumber\\
&=&d_{11}u_{x_1x_1}u_{x_1x_2}+d_{12}(u_{x_1x_2}^2+\mdet(\nabla^2u))+d_{12}u_{x_1x_2}^2+d_{22}u_{x_2x_2}u_{x_1x_2}\nonumber\\
&=&u_{x_1x_2}(u_t-w)+d_{12}\mdet(\nabla^2u),\\
b_{12}&=&b_{21},\\
b_{22}&=&
d_{11}u_{x_1x_2}^2+2d_{12}u_{x_1x_2}u_{x_2x_2}+d_{22}u_{x_2x_2}^2\nonumber\\
&=&d_{11}(u_{x_1x_1}u_{x_2x_2}-\mdet(\nabla^2u))+2d_{12}u_{x_1x_2}u_{x_2x_2}+d_{22}u_{x_2x_2}^2\nonumber\\
&=&u_{x_2x_2}(u_t-w)-d_{11}\mdet(\nabla^2u).
\end{eqnarray*}
That is,
\begin{equation*}
B=(u_t-w)\nabla^2u-\mdet(\nabla^2u)\mdet(D)D^{-1}.
\end{equation*}
Plug this into \eqref{happy3} to obtain
\begin{eqnarray}
D\nabla \vp \cdot\nabla \vp &=&4(D\nabla u)^T\left((u_t-w)\nabla^2u-\mdet(\nabla^2u)\mdet(D)D^{-1}\right)D\nabla u\nonumber\\
&&+4\vp D\nabla^2uD\nabla u\cdot \mathbf{G}+\vp ^2D\mathbf{G}\cdot \mathbf{G}\nonumber\\
&=&-4\vp \mdet(\nabla^2u)\mdet(D)+4D\nabla^2uD\nabla u\cdot (\vp \mathbf{G}+(u_t-w)\nabla u)+\vp ^2D\mathbf{G}\cdot \mathbf{G}.\label{happy4}
\end{eqnarray}
Equipped with the preceding results, we can evaluate
\begin{eqnarray}
\lefteqn{\vp _t-\mdiv\left(D\nabla \vp \right)+\frac{1}{\vp }D\nabla \vp \cdot\nabla \vp }\nonumber\\
&=&2\nabla\mdet(D)\cdot\left(\begin{array}{cc}
u_{x_2x_2}&-u_{x_1x_2}\\
-u_{x_1x_2}&u_{x_1x_1}
\end{array}\right)\nabla u\nonumber\\
&&-2u_t\mdiv\left(D\nabla u\right)-\mdiv\left(\vp D\mathbf{G}-2wD\nabla u\right)+D_t\nabla u\cdot\nabla u\nonumber\\
&&+2\vp ^{-1}D(\nabla \vp -\vp \mathbf{G})\cdot (\vp \mathbf{G}+(u_t-w)\nabla u)+\vp D\mathbf{G}\cdot \mathbf{G}\nonumber\\
&=&2\nabla\mdet(D)\cdot\left(\begin{array}{cc}
u_{x_2x_2}&-u_{x_1x_2}\\
-u_{x_1x_2}&u_{x_1x_1}
\end{array}\right)\nabla u\nonumber\\
&&-2u_t\mdiv\left(D\nabla u\right)-\mdiv\left(\vp D\mathbf{G}-2wD\nabla u\right)+D_t\nabla u\cdot\nabla u\nonumber\\
&&+ (2D\mathbf{G}+2\vp ^{-1}(u_t-w)D\nabla u)\cdot\nabla \vp -2(u_t-w)D\nabla u\cdot \mathbf{G}-\vp D\mathbf{G}\cdot \mathbf{G}
\end{eqnarray}
Subsequently,
%We substitute the preceding equations into \eqref{happy7} to derive
\begin{eqnarray}
\phi_t-\mdiv\left(D\nabla \phi\right)
&=&\frac{1}{\vp }\left(\vp _t-\mdiv\left(D\nabla \vp \right)+\frac{1}{\vp }D\nabla \vp \cdot\nabla \vp \right)\nonumber\\
&=&2\vp ^{-1}\nabla\mdet(D)\cdot\left(\begin{array}{cc}
u_{x_2x_2}&-u_{x_1x_2}\\
-u_{x_1x_2}&u_{x_1x_1}
\end{array}\right)\nabla u\nonumber\\
&&-2\vp ^{-1}u_t\mdiv\left(D\nabla u\right)-\vp ^{-1}\mdiv\left(\vp D\mathbf{G}-2wD\nabla u\right)+\vp ^{-1}D_t\nabla u\cdot\nabla u\nonumber\\
&&+ (2\vp ^{-1}D\mathbf{G}+2\vp ^{-2}(u_t-w)D\nabla u)\cdot\nabla \vp -2\vp ^{-1}(u_t-w)D\nabla u\cdot \mathbf{G}-D\mathbf{G}\cdot \mathbf{G}\nonumber\\
&=&2\vp ^{-1}\nabla\mdet(D)\cdot\left(\begin{array}{cc}
u_{x_2x_2}&-u_{x_1x_2}\\
-u_{x_1x_2}&u_{x_1x_1}
\end{array}\right)\nabla u\nonumber\\
&&-2\vp ^{-1}u_t(u_t+\mdiv D\nabla u-w)-\mdiv\left(D\mathbf{G}-2\vp ^{-1}wD\nabla u\right)+\vp ^{-1}D_t\nabla u\cdot\nabla u\nonumber\\
&&+ (\vp ^{-1}D\mathbf{G}+2\vp ^{-2}u_tD\nabla u)\cdot\nabla \vp -2\vp ^{-1}(u_t-w)D\nabla u\cdot \mathbf{G}-D\mathbf{G}\cdot \mathbf{G}.\label{happy8}
\end{eqnarray}
We still need to eliminate the second partial derivatives of $u$ on the right hand side of the preceding equation.
% If $ \mbox{det}(A)$ had been $1$, then this term would be zero, and hence the proof would conclude. 
%Since we do not have the benefit, we need to continue. 
To this end, we deduce from \eqref{happy5} and \eqref{happy2} that
\begin{eqnarray}
(d_{11}u_{x_1}+d_{12}u_{x_2})u_{x_1x_1}+(d_{12}u_{x_1}+d_{22}u_{x_2})u_{x_1x_2}&=&\frac{1}{2}(\vp _{x_1}-\vp g_1),\label{pp1}\\
(d_{11}u_{x_1}+d_{12}u_{x_2})u_{x_1x_2}+(d_{12}u_{x_1}+d_{22}u_{x_2})u_{x_2x_2}&=&\frac{1}{2}(\vp _{x_2}-\vp g_2),\label{pp2}\\
d_{11}u_{x_1x_1}+2d_{12}u_{x_1x_2}+d_{22}u_{x_2x_2}&=&u_t- w.\label{pp3}
\end{eqnarray}
Denote by $E$ the coefficient matrix of the above system. Then
%The determinant of the coefficient matrix can be computed as follows:
\begin{eqnarray*}
\mbox{det}E&=&\mbox{det}\left(\begin{array}{ccc}
d_{11}u_{x_1}+d_{12}u_{x_2}&d_{12}u_{x_1}+d_{22}u_{x_2}&0\\
0&d_{11}u_{x_1}+d_{12}u_{x_2}&d_{12}u_{x_1}+d_{22}u_{x_2}\\
d_{11}&2d_{12}&d_{22}\end{array}\right)\nonumber\\
&=&(d_{11}u_{x_1}+d_{12}u_{x_2})\left[d_{22}(d_{11}u_{x_1}+d_{12}u_{x_2})-2d_{12}(d_{12}u_{x_1}+d_{22}u_{x_2})\right]\nonumber\\
&&+d_{11}(d_{12}u_{x_1}+d_{22}u_{x_2})^2\nonumber\\
&=&(d_{11}u_{x_1}+d_{12}u_{x_2})\left[(d_{22}d_{11}-2d_{12}^2)u_{x_1}-d_{22}d_{12}u_{x_2}\right]\nonumber\\
&&+d_{11}(d_{12}u_{x_1}+d_{22}u_{x_2})^2\nonumber\\
&=&d_{11}(d_{22}d_{11}-d_{12}^2)u_{x_1}^2 +2(d_{11}d_{22}d_{12}-d_{12}^3)u_{x_1}u_{x_2}+(d_{11}d_{22}^2-d_{22}d_{12}^2)u_{x_2}^2\nonumber\\
&=&\mbox{det}(D)\left(d_{11}u_{x_1}^2 +2d_{12}u_{x_1}u_{x_2}+d_{22}u_{x_2}^2\right)\nonumber\\
&=&\mbox{det}(D)\vp \ne 0.
%\end{array}\right)
\end{eqnarray*}
%Thus the three equations in \eqref{pp1}-\eqref{pp3} are enough
By Cramer's rule, we have
\begin{eqnarray*}
	u_{x_1x_1}
	&=&\frac{1}{2\mbox{det}(D)\vp }\left[((d_{22}d_{11}-2d_{12}^2)u_{x_1}-d_{12}d_{22}u_{x_2})(\vp _{x_1}-\vp g_1)-d_{22}(d_{12}u_{x_1}+d_{22}u_{x_2})(\vp _{x_2}-\vp g_2)\right]\nonumber\\
	&&+\frac{1}{\mbox{det}(D)\vp }(u_t-w)(d_{12}u_{x_1}+d_{22}u_{x_2})^2,\\
	u_{x_1x_2}&=&\frac{1}{2\mbox{det}(D)\vp }\left[d_{11}(d_{12}u_{x_1}+d_{22}u_{x_2}))(\vp _{x_1}-\vp g_1)+d_{22}(d_{11}u_{x_1}+d_{12}u_{x_2}))(\vp _{x_2}-\vp g_2)\right]\\
	&&-\frac{1}{\mbox{det}(D)\vp }(u_t-w)(d_{11}u_{x_1}+d_{12}u_{x_2})(d_{12}u_{x_1}+d_{22}u_{x_2}),\\
	u_{x_2x_2}&=&\frac{-1}{2\mbox{det}(D)\vp }\left[d_{11}(d_{11}u_{x_1}+d_{12}u_{x_2})(\vp _{x_1}-\vp g_1)+(d_{12}d_{11}u_{x_1}-(d_{11}d_{22}-2d_{12}^2)u_{x_2})(\vp _{x_2}-\vp g_2)\right]\nonumber\\
	&&	+\frac{1}{\mbox{det}(D)\vp }(u_t-w)(d_{11}u_{x_1}+d_{12}u_{x_2})^2.
\end{eqnarray*}
Using \eqref{dodef}-\eqref{dthdef} yields
\begin{eqnarray*}
\left(\begin{array}{cc}
-u_{x_2x_2}&u_{x_1x_2}\\
u_{x_1x_2}&-u_{x_1x_1}
\end{array}\right)
&=&\frac{1}{2\mbox{det}(D)\vp }\left(D_1(\nabla \vp -\vp \mathbf{G}), D_2(\nabla \vp -\vp \mathbf{G})\right)+\frac{(u_t-w)}{\mbox{det}(D)\vp }D_3.
\end{eqnarray*}
Observe that 
\begin{eqnarray*}
\nabla\mdet(D)\cdot(D_1\nabla \vp , D_2\nabla  \vp )\nabla u&=&(\nabla u)^T\left(\begin{array}{c}
(\nabla \vp )^TD_1^T\nabla\mdet(D)\\
(\nabla \vp )^TD_2^T\nabla\mdet(D)
\end{array}\right)\nonumber\\
&=&(\nabla u)^T\left(\begin{array}{c}
(\nabla\mdet(D))^TD_1\nabla \vp \\
(\nabla\mdet(D))^TD_2\nabla \vp 
\end{array}\right)\nonumber\\
&=&\left(
D_1^T\nabla\mdet(D),
D_2^T\nabla\mdet(D)
\right)\nabla u\cdot\nabla \vp .
\end{eqnarray*}
With this in mind, we derive from \eqref{happy8} that
\begin{eqnarray*}
\lefteqn{\phi_t-\mbox{div}(D\nabla \phi)}\nonumber\\
&=&2\vp ^{-1}\nabla\mdet(D)\cdot\left(\begin{array}{cc}
u_{x_2x_2}&-u_{x_1x_2}\\
-u_{x_1x_2}&u_{x_1x_1}
\end{array}\right)\nabla u\nonumber\\
&&-2\vp ^{-1}u_t(u_t+\mdiv D\nabla u-w)-\mdiv\left(D\mathbf{G}-2\vp ^{-1}wD\nabla u\right)+\vp ^{-1}D_t\nabla u\cdot\nabla u\nonumber\\
&&+ (\vp ^{-1}D\mathbf{G}+2\vp ^{-2}u_tD\nabla u)\cdot\nabla \vp -2\vp ^{-1}(u_t-w)D\nabla u\cdot \mathbf{G}-D\mathbf{G}\cdot \mathbf{G}\nonumber\\
&=&2\vp ^{-1}\nabla\mdet(D)\cdot\left(\frac{1}{2\mbox{det}(D)\vp }\left(D_1(\nabla \vp -\vp \mathbf{G}), D_2(\nabla \vp -\vp \mathbf{G})\right)+\frac{(u_t-w)}{\mbox{det}(D)\vp }D_3\right)\nabla u\nonumber\\
&&-2\vp ^{-1}u_t(u_t+\mdiv D\nabla u-w)-\mdiv\left(D\mathbf{G}-2\vp ^{-1}wD\nabla u\right)+\vp ^{-1}D_t\nabla u\cdot\nabla u\nonumber\\
&&+ (\vp ^{-1}D\mathbf{G}+2\vp ^{-2}u_tD\nabla u)\cdot\nabla \vp -2\vp ^{-1}(u_t-w)D\nabla u\cdot \mathbf{G}-D\mathbf{G}\cdot \mathbf{G}\nonumber\\
&=&\frac{1}{\mbox{det}(D)\vp ^2}\nabla\mdet(D)\cdot\left(D_1\nabla \vp , D_2\nabla \vp \right)\nabla u-\frac{1}{\mbox{det}(D)\vp }\nabla\mdet(D)\cdot\left(D_1\mathbf{G}, D_2\mathbf{G}\right)\nabla u\nonumber\\
&&+\frac{2(u_t-w)}{\mbox{det}(D)\vp ^2}\nabla\mdet(D)\cdot D_3\nabla u-2\vp ^{-1}u_t(u_t+\mdiv D\nabla u-w)-\mdiv\left(D\mathbf{G}-2\vp ^{-1}wD\nabla u\right)\nonumber\\
&&+\vp ^{-1}D_t\nabla u\cdot\nabla u+ (\vp ^{-1}D\mathbf{G}+2\vp ^{-2}u_tD\nabla u)\cdot\nabla \vp -2\vp ^{-1}(u_t-w)D\nabla u\cdot \mathbf{G}-D\mathbf{G}\cdot \mathbf{G}\nonumber\\
&=&\left(D\mathbf{G}+2\vp ^{-1}u_tD\nabla u+\frac{1}{\mbox{det}(D)\vp }\left(
D_1^T\nabla\mdet(D),
D_2^T\nabla\mdet(D)
\right)\nabla u\right)\cdot\nabla \phi\nonumber\\
&&-\frac{1}{\mbox{det}(D)\vp }\nabla\mdet(D)\cdot\left(D_1\mathbf{G}, D_2\mathbf{G}\right)\nabla u+\frac{2(u_t-w)}{\mbox{det}(D)\vp ^2}\nabla\mdet(D)\cdot D_3\nabla u\nonumber\\
&&-\mdiv\left(D\mathbf{G}-2\vp ^{-1}wD\nabla u\right)-2\vp ^{-1}u_t(u_t+\mdiv D\nabla u-w)\nonumber\\
&&+\vp ^{-1}D_t\nabla u\cdot\nabla u-2\vp ^{-1}(u_t-w)D\nabla u\cdot \mathbf{G}-D\mathbf{G}\cdot \mathbf{G}.
\end{eqnarray*}
Using \eqref{hdef}-\eqref{shdef}, we have
\begin{equation}\label{eqfb}
\phi_t-\mbox{div}(D\nabla \phi)=\mathbf{H}\cdot\nabla \phi+h+\mdiv\mathbf{F}.
\end{equation}
To see \eqref{wine111}, we compute
\begin{eqnarray*}
\phi_t&=& \vp ^{-1}\vp _t=\frac{1}{j}\vp ^{-1}\psi^{\frac{1}{j}-1}\psi_t=\frac{1}{j}\psi^{-1}\psi_t,\\
\nabla \phi&=&\vp ^{-1}\nabla \vp =\frac{1}{j}\vp ^{-1}\psi^{\frac{1}{j}-1}\nabla\psi=\frac{1}{j}\psi^{-1}\nabla\psi.
\end{eqnarray*}
Substituting these into \eqref{eqfb} gives the desired result. The proof is complete.
\end{proof}

\section{A priori estimates}
In this section we derive a priori estimates for solutions to \eqref{gwe1}-\eqref{gwe5}. The main theorem will be established as a consequence of these estimates. We begin with the energy estimate.
\begin{lemma}\label{ue} Assume $u_0\in L^2(\Omega)$. Then we have
	\begin{eqnarray*}
\frac{1}{2}\sup_{0\leq t\leq T}\io u^2dx+\ioT D\nabla u\cdot\nabla udxdt&\leq & \io u_0^2(x)dx.
	%,\\	\sup_{0\leq t\leq T}
%	\sup_{\ot}|u|&\leq & c.
	\end{eqnarray*}
	\end{lemma}
\begin{proof}Use $u$ as a test function in \eqref{gwe2} and keep in mind \eqref{ellip} to obtain
	\begin{equation}\label{ue1}
	\frac{1}{2}\frac{d}{dt}\io u^2dx+\io\left((a|\q|+m)|\nabla u|^2+\frac{b-a}{|\q|}(\q\cdot\nabla u)^2\right)dx=-\frac{1}{2}\io\q\cdot\nabla u^2dx.
	\end{equation}
By the definition of $\q$, we have
\begin{equation*}
\io\q\cdot\nabla u^2dx=\io\left(-\vxt (u^2)_{x_1}+\vxo (u^2)_{x_2}\right)dx=\io\mdiv\left(\begin{array}{c}
v(u^2)_{x_2}\\
-v(u^2)_{x_1}
\end{array}\right)dx=0.
\end{equation*}	
The last step is due to the boundary condition for $v$.	Plug this into \eqref{ue1} and integrate to obtain the desired result.	\end{proof}

\begin{lemma}\label{ump}
	The function $u$ satisfies the weak maximum principle, i.e.,
	\begin{equation*}
	\sup_{\ot}|u|\leq \|u_0\|_{\infty,\Omega}.
	\end{equation*}
\end{lemma}
This is a consequence of \eqref{diq}.
Indeed, let $L=\|u_0\|_{\infty,\Omega}$. Then we can write \eqref{gwe2} in the form
\begin{equation*}
u_t-\mdiv\left(D\nabla u-(u-L)\q\right)= 0\ \ \ \mbox{in $\ot$},
\end{equation*}
where $D$ is given as in \eqref{ddf}. Use $(u-L)^+$ as a test function in this equation and then apply \eqref{ellip} to get
\begin{eqnarray*}
\lefteqn{\frac{1}{2}\frac{d}{dt}\io\left[(u-L)^+\right]^2dx}\nonumber\\
&&+\io \left((a|\mathbf{q}|+m)|\nabla(u-L)^+|^2+(b-a)\frac{(\q\cdot\nabla(u-L)^+)^2}{|\q|}\right)dx\nonumber\\
&=&-\frac{1}{2}\io\q\cdot\nabla\left[(u-L)^+\right]^2dx.
\end{eqnarray*}
As before, the last integral is zero. Thus we have
\begin{equation*}
\frac{1}{2}\frac{d}{dt}\io\left[(u-L)^+\right]^2dx+m\io|\nabla(u-L)^+|^2dx\leq 0.
\end{equation*}
Integrate to obtain the desired result.

According to a result in (\cite{R}, p.82), for each $r>1$ there is a positive number $c$ such that
\begin{equation}\label{qlpb}
\sup_{0\leq t\leq T}\|\q\|_{r,\Omega}=\sup_{0\leq t\leq T}\|\nabla v\|_{r,\Omega}\leq c\sup_{0\leq t\leq T}\|u\|_{r,\Omega}\leq c.
\end{equation}
In addition, we can conclude from the classical Calder\'{o}n-Zygmund estimate that
\begin{equation}\label{vwtb}
\int_{0}^{T}\|v\|_{W^{2,2}(\Omega)}^2dt\leq c\int_{0}^{T}\io(\uxo)^2dxdt\leq c.
\end{equation}

Now we set
\begin{equation*}
\beta=v_t,\ \ \omega= u_t.
\end{equation*}
%It is easy to check that
By \eqref{dpt} and \eqref{qd}, we have
\begin{equation}\label{dtb}
|D_t|\leq c|\nabla\beta|,\ \ \ |\q_t|=|\nabla\beta|.
\end{equation}
Differentiate \eqref{gwe1} with respect to $t$ to get
\begin{equation*}
\Delta\beta=\omega_{x_1}\ \ \ \mbox{in $\ot$.}
\end{equation*}
Moreover,
\begin{equation*}
\beta\mid_{\po}=0.
\end{equation*}
%A result of \cite{R} asserts that 
On account of \eqref{qlpb}, for each $r>1$ there is a positive number $c$
such that
\begin{equation}\label{pplr}
\|\nabla\beta\|_{r, \Omega}\leq c\|\omega\|_{r,\Omega}.
\end{equation}
For Theorem \ref{keyth} to be useful to us, we must be able to bound $u_t$ by $\nabla u$ . The following two lemmas address this issue.
\begin{lemma}\label{vptb1} We have
	\begin{equation*}
	\ioT u_t^2dxdt+\sup_{0\leq t\leq T}\io D\nabla u\cdot\nabla u dx\leq c\ioT|\nabla u|^4dxdt+c.
	\end{equation*}
\end{lemma}
\begin{proof}
Use $u_t$ as a test function in \eqref{gwe2} to obtain
\begin{equation}\label{utb}
\io (u_t)^2dx+\io D\nabla u\cdot\nabla u_t=-\io \q\cdot\nabla u u_t dx.
\end{equation}
Note that $D$ is symmetric. Hence
\begin{eqnarray*}
%\io u\q\cdot\nabla u_t dx&=&-\io\nabla u\cdot\q u_tdx,\\
\io D\nabla u\cdot\nabla u_t&=&\frac{1}{2}\frac{d}{dt}\io D\nabla u\cdot\nabla udx-\frac{1}{2}\io D_t\nabla u\cdot\nabla udx.
\end{eqnarray*}
%The last step is due to the fact that 
We estimate from \eqref{dtb} and \eqref{pplr} that
\begin{eqnarray*}
\io D_t\nabla u\cdot\nabla udx&\leq &\left(\io |D_t|^2dx\right)^{\frac{1}{2}}\left(\io|\nabla u|^4dx\right)^{\frac{1}{2}}\nonumber\\
&\leq & c\left(\io |u_t|^2dx\right)^{\frac{1}{2}}\left(\io|\nabla u|^4dx\right)^{\frac{1}{2}}.
\end{eqnarray*}
Similarly,
\begin{eqnarray*}
-\io \q\cdot\nabla u u_t dx&\leq &\left(\io |u_t|^2dx\right)^{\frac{1}{2}}\left(\io|\nabla u|^4dx\right)^{\frac{1}{4}}\left(\io|\q|^4dx\right)^{\frac{1}{4}}\nonumber\\
&\leq &c\left(\io |u_t|^2dx\right)^{\frac{1}{2}}\left(\io|\nabla u|^4dx\right)^{\frac{1}{4}}.
\end{eqnarray*}
Here the last step is due to \eqref{qlpb}.
Use the above two inequalities in \eqref{utb} and integrate to derive
\begin{equation*}
\ioT(u_t)^2dxdt+\sup_{0\leq t\leq T}\io  D\nabla u\cdot\nabla udx\leq c\ioT|\nabla u|^4dxdt+c\io D(x,0)\nabla u_0\cdot\nabla u_0dx+c.
\end{equation*}
Our assumptions on $u_0$ implies that 
\begin{equation*}
\io D(x,0)\nabla u_0\cdot\nabla u_0dx\leq c.
\end{equation*}
To see this, remember that $v_0\equiv v(x,0)$ is the solution of
\begin{eqnarray}
\Delta v_0&=&(u_0)_{x_1}\ \ \mbox{in $\Omega$},\label{vz1}\\
v_0&=&0\ \ \ \mbox{on $\po$}\label{vzt}
\end{eqnarray}
and $$\q(x,0)=\left(\begin{array}{c}
-(v_0)_{x_2}\\
(v_0)_{x_1}
\end{array}\right).$$
In fact, it is enough for us to assume that $u_0\in W^{1,p}(\Omega)$ for some $p>2$ because this already implies $|\q(x,0)|=|\nabla v_0|\in L^\infty(\Omega)$.
Therefore,
\begin{eqnarray}
|D(x,0)|&=&|(a|\q(x,0)|+m)I+\frac{b-a}{|\q(x,0)|}\q(x,0)\otimes\q(x,0)|\nonumber\\
&\leq &b|\q(x,0)|+m\leq c.\label{dzd}
\end{eqnarray}
%The matrix $D(x,0)$ is obtained as follows: We first set $t=0$ in \eqref{gwe1} to solve for $v_0\equiv v(x,0)$. That is,
%with 
%Then we form 
The proof is complete.
\end{proof}
%Now we proceed to derive estimates for $  $. For this purpose, we %define
% Write \eqref{gwe2} in the formFor each $\ell>4$
%\begin{equation}
%u_t-\mdiv\left(D\nabla u\right)= -\q\cdot\nabla u\ \ \ \mbox{in $\ot$}.
%\end{equation}
\begin{lemma}\label{plin} 
	For each $q>2$ there is a positive number $c$ such
	that
%	We have
	\begin{equation*}
	\|\omega\|_{\infty,\ot}\leq c\|\nabla u\|_{\infty,\ot}^{2+q}+c.
	\end{equation*}
	\end{lemma}
\begin{proof}First we observe that  
	\begin{equation}\label{pic}
	\|\omega(x,0)\|_{\infty,\Omega}\leq c.
	\end{equation}
	To see this, we let $t=0$ in \eqref{gwe2} to obtain
	\begin{eqnarray*}
	\omega(x,0)&=&\mdiv\left(D(x,0)\nabla u(x,0)\right)-\nabla u(x,0)\q(x,0)\nonumber\\
	&=&D(x,0):\nabla^2u_0+\mdiv D(x,0)\nabla u_0-\nabla u_0\q(x,0).
	\end{eqnarray*}
We can easily derive \eqref{pic} from (H2), \eqref{dzd}, \eqref{dpx}, and \eqref{vz1}.
	
%	Our assumptions on $u_0$ implies that $\omega(x,0)\in L^\infty(\Omega)$.
Differentiate \eqref{gwe2} with respect to $t$ to get
\begin{equation}\label{wine11}
\omega_t-\mdiv\left(D\nabla \omega+D_{t}\nabla u\right)=-\q\cdot\nabla\omega-\q_t\cdot\nabla u \ \ \ \mbox{in $\ot$}.
\end{equation}
Furthermore, 
\begin{eqnarray}
%\io\omega dx&=&0,\label{vpa}\\
\partial_t\left(D\nabla u\right)\cdot\nu&=&0\ \ \mbox{on $\Sigma_T$.}\label{vptb}
\end{eqnarray}
Pick 
\begin{eqnarray}
k\geq 2\|\omega(x,0)\|_{\infty,\Omega}\label{cnk}
\end{eqnarray} as below. Set
\begin{equation*}
k_n=k-\frac{k}{2^{n+1}},\ \  n=0,1,2,\cdots.
\end{equation*}
Assume that
\begin{equation*}
\sup_{\ot}\omega=\|\omega\|_{\infty, \ot}.
\end{equation*}
Otherwise, consider $-\omega$. In view of \eqref{vptb}, we can use $(\omega-k_n)^+$ as a test function in \eqref{wine11} to obtain
%equation and then apply \eqref{ellip} to get
\begin{eqnarray}
\lefteqn{\frac{1}{2}\frac{d}{dt}\io\left[(\omega-k_n)^+\right]^2dx+\io D\nabla(\omega-k_n)^+\cdot\nabla(\omega-k_n)^+dx }\nonumber\\
&=&-\frac{1}{2}\io\q\cdot\nabla\left[(\omega-k_n)^+\right]^2dx-\io D_{t}\nabla u\cdot\nabla(\omega-k_n)^+dx-\io \nabla u\cdot\q_t(\omega-k_n)^+dx\nonumber\\
&=&-\io D_{t}\nabla u\cdot\nabla(\omega-k_n)^+dx-\io \nabla u\cdot\q_t(\omega-k_n)^+dx.\label{pinb1}
\end{eqnarray}
Set
\begin{eqnarray*}
\Omega_{n}(t)&=&\{x\in\Omega: \omega(x,t)\geq k_n\},\\
A_{\Omega,k}(t)&=&\frac{1}{|\Omega|}\io(\omega(x,t)-k_n)^+dx.
\end{eqnarray*}
By virtue of Poincar\'{e}'s inequality, for each $r>2$ we have that 
\begin{eqnarray*}
-\io \nabla u\cdot\q_t(\omega-k_n)^+dx&\leq &\|\nabla u\cdot\q_t\|_{\frac{r}{r-1}, \Omega_{n}(t)}\left(\io|\nabla(\omega-k_n)^+|^{\frac{2r}{r+2}}dx\right)^{\frac{r+2}{2r}}\nonumber\\
&&-A_{\Omega,k}(t)\int_{\Omega_{n}(t)} \nabla u\cdot\q_tdx\nonumber\\
&\leq &c\|\nabla u\cdot\q_t\|_{\frac{r}{r-1}, \Omega_{n}(t)}\left(\io|\nabla(\omega-k_n)^+|^{2}dx\right)^{\frac{1}{2}}\nonumber\\
&&+c\left(\io\left[(\omega-k_n)^+\right]^2dx\right)^{\frac{1}{2}}\int_{\Omega_{n}(t)} |\nabla u\cdot\q_t|dx.
\end{eqnarray*}
%from whence follows
This together with \eqref{pinb1} implies
\begin{equation*}
\sup_{0\leq t\leq T}\io\left[(\omega-k_n)^+\right]^2dx+\ioT |\nabla(\omega-k_n)^+|^2dxdt\leq c \int_{0}^{T}\int_{\Omega_{n}(t)}|D_{t}\nabla u-u\q_t|^2dxdt.
\end{equation*}
Here we have used the fact that $\frac{r}{r-1}<2$.
Let
\begin{equation*}
Y_n=\int_{0}^{T}|\Omega_{n}(t)|dt=|\{(x,t)\in\ot:\omega(x,t)\geq k_n\}|.
\end{equation*}
Let $s\in (1,2)$ be given.  We estimate from Poincar\'{e}'s inequality that
\begin{eqnarray}
\lefteqn{\ioT\left[(\omega-k_n)^+\right]^{2s}dxdt}\nonumber\\
&=&\int_{0}^{T}\left(\io\left[(\omega-k_n)^+\right]^2dx\right)^{\frac{s}{2}}\left(\io\left[(\omega-k_n)^+\right]^{\frac{2s}{2-s}}dx\right)^{\frac{2-s}{2}}dt\nonumber\\
&\leq &2^{s-1}\left(\sup_{0\leq t\leq T}\io\left[(\omega-k_n)^+\right]^2dx\right)^{\frac{s}{2}}\int_{0}^{T}\left(\io|(\omega-k_n)^+-A_{\Omega,k}(t)|^{\frac{2s}{2-s}}dx\right)^{\frac{2-s}{2}}dt\nonumber\\
&&+2^{s-1}\left(\sup_{0\leq t\leq T}\io\left[(\omega-k_n)^+\right]^2dx\right)^{\frac{s}{2}}\int_{0}^{T}A_{\Omega,k}^s(t)|\Omega_n(t)|^{\frac{2-s}{2}}dt\nonumber\\
&\leq &\left(\sup_{0\leq t\leq T}\io\left[(\omega-k_n)^+\right]^2dx\right)^{\frac{s}{2}}\int_{0}^{T}\io|\nabla(\omega-k_n)^+|^{s}dxdt\nonumber\\
&&+c\left(\sup_{0\leq t\leq T}\io\left[(\omega-k_n)^+\right]^2dx\right)^{\frac{s}{2}}\left(\int_{0}^{T}A_{\Omega,k}^2(t)dt\right)^{\frac{s}{2}}Y_n^{\frac{2-s}{2}}\nonumber\\
&\leq &c\left(\sup_{0\leq t\leq T}\io\left[(\omega-k_n)^+\right]^2dx\right)^{\frac{s}{2}}\left(\ioT|\nabla(\omega-k_n)^+|^{2}dxdt\right)^{\frac{s}{2}}Y_n^{1-\frac{s}{2}}\nonumber\\
&&+c\left(\sup_{0\leq t\leq T}\io\left[(\omega-k_n)^+\right]^2dx\right)^{s}Y_n^{\frac{2-s}{2}}\nonumber\\
&\leq &c\left(\int_{0}^{T}\int_{\Omega_{n}(t)}|D_{t}\nabla u-u\q_t|^2dxdt\right)^sY_n^{1-\frac{s}{2}}\nonumber\\
&\leq &c\|D_{t}\nabla u-u\q_t\|_{2q,\ot}^{2s}Y_n^{1+\frac{s(q-2)}{2q}},\ \ \ \label{vpt1}
\end{eqnarray}
where $q>2$ is given as in the lemma.
On the other hand, we have
\begin{equation*}
\ioT\left[(\omega-k_n)^+\right]^{2s}dxdt\geq (k_{n+1}-k_n)^{2s}Y_{n+1}=\frac{k^{2s}}{4^{s(n+1)}}Y_{n+1}.
\end{equation*}
Combining this with \eqref{vpt1} yields
\begin{equation*}
Y_{n+1}\leq \frac{c4^{sn}\|D_{t}\nabla u-u\q_t\|_{2q,\ot}^{2s}}{k^{2s}}Y_n^{1+\frac{s(q-2)}{2q}}.
\end{equation*}
By Lemma \ref{ynb},
%4.1 in (\cite{D}, p. 12), 
we have
\begin{equation}\label{r1}
\lim_{n\rightarrow \infty}Y_n=0,
\end{equation}
provided that
\begin{equation*}
Y_0\leq c\left(\frac{k^{2s}}{\|D_{t}\nabla u-u\q_t\|_{2q,\ot}^{2s}}\right)^{\frac{2q}{s(q-2)}}.
\end{equation*}
%This together with \eqref{cnk} implies
In view of \eqref{cnk} and \eqref{r1}, if we take $k=c\|D_{t}\nabla u-u\q_t\|_{2q,\ot}+2\|\omega(x,0)\|_{\infty,\Omega}$, then
\begin{equation*}
\omega\leq k.
%=c\|D_{t}\nabla u-u\q_t\|_{2q,\ot}+c\|\omega(x,0)\|_{\infty,\Omega}.
\end{equation*}
Subsequently,
\begin{eqnarray*}
\|\omega\|_{\infty,\ot}&\leq &c\|\nabla u\|_{\infty,\ot}\|D_{t}\|_{2q,\ot}+c\|\q_t\|_{2q,\ot}+c\nonumber\\
&\leq &c(\|\nabla u\|_{\infty,\ot}+1)\|\nabla \beta\|_{2q,\ot}+c\nonumber\\
&\leq &c(\|\nabla u\|_{\infty,\ot}+1)\|\omega\|_{2q,\ot}+c.
\end{eqnarray*}
The last step is due to \eqref{pplr}.
%
%Fix $\ell>2q$. We estimate from \eqref{pplr} that
%\begin{eqnarray}
%\|\omega\|_{\infty,\ot}&\leq &c\|\nabla w\|_{\ell,\ot}\|\nabla u\|_{\frac{2q\ell}{\ell-2q},\ot}+c\|\nabla w\|_{2q,\ot}+c\nonumber\\
%&\leq &c\|\omega\|_{\ell,\ot}\|\nabla u\|_{\frac{2q\ell}{\ell-2q},\ot}+c\|\omega\|_{2q,\ot}+c
%\end{eqnarray}
In view of the interpolation inequality (\cite{GT}, p.146), we have
\begin{equation*}
\|\omega\|_{2q,\ot}\leq \varepsilon \|\omega\|_{\infty,\ot}+\frac{1}{\varepsilon^{q-1}}\|\omega\|_{2,\ot}.
\end{equation*}
By choosing $\varepsilon$ suitably, we arrive
at
\begin{eqnarray*}
\|\omega\|_{\infty,\ot}&\leq &c\|\omega\|_{2,\ot}(\|\nabla u\|_{\infty,\ot}+1)^{q}+c.
%\nonumber\\
%&\leq &c\|\omega\|_{2,\ot}(\|\nabla u\|_{\infty,\ot}+1)^{2}+c.
\end{eqnarray*}
%The last step is due to the fact that $\frac{q}{q-1}<2$ since $q>2$.
Use Lemma \ref{vptb1} to yield the desired result. 
%\begin{equation*}
\end{proof}

We are ready to prove the main theorem
\begin{proof}[Proof of the Main Theorem]
By Theorem \ref{keyth}, the function $\psi=\vp^j=(D\nabla u\cdot\nabla u)^j$ satisfies
\begin{equation}\label{beer11}
\frac{1}{\psi}\psi_t-\textup{div}\left(\frac{1}{\psi}D\nabla\psi\right)= \frac{1}{\psi}\mathbf{H}\cdot\nabla\psi+jh+j\textup{div}\mathbf{F}\ \ \mbox{in $\{|\nabla u|>0\}$}
\end{equation}
%The proof of this theorem involves tons of calculations \cite{X7}.
% Now fix a point $z_0=(x_0,t_0)\in \ot$. Then pick a number $R$ from $(0,\min\{\mbox{dist}(x_0,\partial\Omega),\sqrt{t_0}\})$. Define a sequence of cylinders $Q_{R_n}(z_0)$ in $\ot$ as follows:
 \begin{equation*}
 Q_{R_n}(z_0)=B_{R_n}(x_0)\times(t_0-R_n^2, t_0],
 \end{equation*}
 where
 \begin{equation*}
 R_n=\frac{ R}{2}+\frac{R}{2^{n+1}}\ \,\ n=0,1,2,\cdots.
 %R_n=\frac{R}{2}+\frac{R}{2^{n+1}},\ \ \ n=0,1,2,\cdots. \mbox{for some $\sigma\in (0,1)$}
 \end{equation*}
  Choose a sequence of smooth functions $\theta_n$ so that
 \begin{eqnarray*}
 \theta_n(x,t)&=& 1 \ \ \mbox{in $Q_{R_n}(z_0)$},\\
 \theta_n(x,t)&=&0\ \ \mbox{outside $B_{R_{n-1}}(x_0)$ and $t<t_0-R_n^2$},\\
 |\partial_t\theta_n(x,t)&|\leq &\frac{c4^n}{R^2}\ \ \mbox{on $Q_{R_{n-1}}(z_0)$},\\
 |\nabla \theta_n(x,t)|&\leq & \frac{c2^n}{R}\ \ \mbox{on $Q_{R_{n-1}}(z_0)$,}\ \ \ \mbox{and}\\
 0&\leq &\theta_n(x,t)\leq 1\ \ \mbox{on $Q_{R_{n-1}}(z_0)$.}
 \end{eqnarray*}
 %Let $K$ be a positive number to be determined. 
 Let
 \begin{equation*}
 Q_0=\{(x,t)\in Q_{R_0}(z_0):\psi(x,t)\geq 1\}.
 \end{equation*}
Select $\frac{K}{2}\geq 1$
 as below.
 Set
 \begin{equation*}
 K_n=K-\frac{K}{2^{n+1}},\ \ \ n=0,1,2,\cdots.
 \end{equation*}
 It follows that
% Subsequently,
 \begin{equation}\label{conk}
 1\leq\frac{K}{2}\leq K_n\leq K.
 \end{equation}
 We use $\theta_{n+1}^2(\psi-K_{n+1})^+$  as a test function in \eqref{beer11} to obtain
\begin{eqnarray}
\lefteqn{\frac{d}{dt}\io\int_{0}^{\ln\psi}(e^s-K_{n+1})^+ds\theta_{n+1}^2dx+\io\frac{1}{\psi} D\nabla \psi\cdot\nabla (\psi-K_{n+1})^+\theta_{n+1}^2dx}\nonumber\\
&=&2\io\int_{0}^{\ln\psi}(e^s-K_{n+1})^+ds\theta_{n+1}\partial_t\theta_{n+1}dx-2\io\frac{1}{\psi}  D\nabla \psi\cdot\nabla\theta_{n+1}(\psi-K_{n+1})^+\theta_{n+1}dx\nonumber\\
&&+\io\frac{1}{\psi}\mathbf{H}\cdot\nabla\psi\theta_{n+1}^2(\psi-K_{n+1})^+dx
+j\io h\theta_{n+1}^2(\psi-K_{n+1})^+dx\nonumber\\
&&-j\io \mathbf{F}\cdot\nabla (\psi-K_{n+1})^+\theta_{n+1}^2dx-2j\io\mathbf{F} \cdot\nabla\theta_{n+1}(\psi-K_{n+1})^+\theta_{n+1}dx.\label{wine6}
\end{eqnarray}
We easily evaluate
\begin{equation*}
\int_{0}^{\ln\psi}(e^s-K_{n+1})^+ds=(\psi-K_{n+1})^+-K_{n+1}(\ln\psi-\ln K_{n+1})^+.
\end{equation*}
We claim that
\begin{equation}\label{haha1}
(\psi-K_{n+1})^+-K_{n+1}(\ln\psi-\ln K_{n+1})^+\geq \left[\left(\sqrt{\psi}-\sqrt{K_{n+1}}\right)^+\right]^2.
\end{equation}
To see this, we consider the function
\begin{equation*}
g(s)=\frac{2}{\sqrt{K_{n+1}}}\left(\sqrt{s}-\sqrt{K_{n+1}}\right)-\ln s+\ln K_{n+1}\ \ \mbox{on $[K_{n+1},\infty)$.}
\end{equation*}
A simple calculation shows that
\begin{equation*}
g(K_{n+1})= 0 \ \ \mbox{and}\ \
g^\prime(s)=\frac{1}{\sqrt{s}}\left(\frac{1}{\sqrt{K_{n+1}}}-\frac{1}{\sqrt{s}}\right)> 0 \ \ \mbox{for $s>K_{n+1}$.}
\end{equation*}
This immediately implies that
\begin{equation*}
\frac{2}{\sqrt{K_{n+1}}}\left(\sqrt{\psi}-\sqrt{K_{n+1}}\right)^+\geq \left(\ln \psi-\ln K_{n+1}\right)^+.
\end{equation*}
It is not difficult to see that this inequality is equivalent to \eqref{haha1}.

Set
\begin{eqnarray}
S_{n+1}(t)&=&\{x\in B_n(x_0): \psi(x, t)\geq K_{n+1}\},\\
Q_{n+1}&=&\{(x,t)\in Q_{R_n}(z_0), \psi(x,t)\geq K_{n+1}\}=\cup_{0\leq t\leq T}S_{n+1}(t)\subset Q_{R}(z_0).
\end{eqnarray}
%Note that
Then we have
\begin{eqnarray*}
\nabla\psi&=&\nabla (\psi-K_{n+1})^+\ \ \mbox{on $S_{n+1}(t)$},\\
|Q_{n+1}|&=&\int_{t_0-R_n^2}^{t_0}|S_{n+1}(t)|dt.
\end{eqnarray*}
By \eqref{ellip}, we have
\begin{eqnarray*}
D\nabla \psi\cdot\nabla (\psi-K_{n+1})^+&=&(a|\q|+m)|\nabla (\psi-K_{n+1})^+|^2+\frac{b-a}{|\q|}(\q\cdot\nabla (\psi-K_{n+1})^+)^2,\\
 D\nabla \psi(\psi-K_{n+1})^+&=&(a|\q|+m)\nabla (\psi-K_{n+1})^+(\psi-K_{n+1})^+\nonumber\\
 &&+\frac{b-a}{|\q|}(\q\cdot\nabla (\psi-K_{n+1})^+)(\psi-K_{n+1})^+\q.
\end{eqnarray*}
Integrate \eqref{wine6} with respect to t and then incorporate \eqref{haha1} and the preceding two equations into the resulting equation to deduce
%This combined with  implies
\begin{eqnarray}
\lefteqn{\io\left[\left(\sqrt{\psi}-\sqrt{K_{n+1}}\right)^+\right]^2\theta_{n+1}^2dx+ \iot\frac{1}{\psi}(a|\q|+m)|\nabla (\psi-K_{n+1})^+|^2 \theta_{n+1}^2dxdt}\nonumber\\
&\leq&\frac{c4^n}{R^2}\int_{Q_{R_n}(z_0)}(\psi-K_{n+1})^+dxdt+\frac{c4^n}{R^2}\int_{Q_{R_n}(z_0)}\frac{(a|\q|+m)}{\psi} \left[(\psi-K_{n+1})^+\right]^2dxdt\nonumber\\
&&+\iot\frac{c}{\psi}|\mathbf{H}|^2\theta_{n+1}^2\left[(\psi-K_{n+1})^+\right]^2dxdt
+c\iot| h|\theta_{n+1}^2(\psi-K_{n+1})^+dxdt\nonumber\\
&&+c\int_{t_0-R_n^2}^{\tau}\int_{S_{n+1}(t)} \psi|\mathbf{F}|^2\theta_{n+1}^2dxdt+\frac{c2^n}{R}\iot|\mathbf{F}| (\psi-K_{n+1})^+\theta_{n+1}dxdt,\label{wine12}
\end{eqnarray}
where $\Omega_\tau=\Omega\times(0,\tau)$ for $\tau\in (0, t_0]$.
The last term in \eqref{wine12} can be absorbed into the remaining terms on the right-hand side of \eqref{wine12} because we have the following estimate
\begin{eqnarray}
%\int_{Q_{R_n}(z_0)}(\psi-K_{n+1})^+dxdt&\leq &\frac{1}{2}\int_{t_0-R_n^2}^{t_0}\int_{S_{n+1}(t)}\psi dxdt+\frac{1}{2}\int_{Q_{R_n}(z_0)}\frac{\left[(\psi-K_{n+1})^+\right]^2}{\psi} dxdt,\label{r13}\\
\frac{2^n}{R}\iot|\mathbf{F}| (\psi-K_{n+1})^+\theta_{n+1}dxdt&\leq& \frac{c4^n}{R^2}\int_{t_0-R_n^2}^{\tau}\int_{S_{n+1}(t)}\frac{1}{\psi} \left[(\psi-K_{n+1})^+\right]^2dxdt\nonumber\\
&&+c\int_{t_0-R_n^2}^{\tau}\int_{S_{n+1}(t)} \psi|\mathbf{F}|^2\theta_{n+1}^2dxdt,\ \ \tau\in(t_0-R_n^2, t_0).\label{r12}
\end{eqnarray}
Set 
\begin{equation}
w=\theta_{n+1}\left(\sqrt{\psi}-\frac{K_{n+1}}{\sqrt{\psi}}\right)^+.
\end{equation}
Thus
\begin{eqnarray}
w&\leq&\theta_{n+1}\sqrt{\psi}\chi_{Q_{n+1}},\label{r15}\\
|\nabla w|&\leq &\left|\theta_{n+1}\left(\frac{1}{2}\psi^{-\frac{1}{2}}+\frac{K_{n+1}}{2\psi}\psi^{-\frac{1}{2}}\right)\nabla\psi\right|\chi_{Q_{n+1}}+\frac{c2^n}{R}\sqrt{\psi}\chi_{Q_{n+1}}\nonumber\\
&\leq&\theta_{n+1}\psi^{-\frac{1}{2}}\left|\nabla\psi\right|\chi_{Q_{n+1}}+\frac{c2^n}{R}\sqrt{\psi}\chi_{Q_{n+1}}.\label{r17}
\end{eqnarray}
Here we have used the fact that $\psi\geq K_{n+1}$ on $Q_{n+1}$.
By virtue of \eqref{kin} and \eqref{kcl}, we derive
\begin{eqnarray}
\lefteqn{
	\int_{B_{R_n}(x_0)}\frac{1}{\psi}|\mathbf{H}|^2\theta_{n+1}^2\left[(\psi-K_{n+1})^+\right]^2dx
=\int_{B_{R_n}(x_0)}|\mathbf{H}|^2w^2dx}\nonumber\\
&\leq &\eta(|\mathbf{H}|^2; R_n; B_{R_n}(x_0))\left(\int_{B_{R_n}(x_0)}|\nabla w|^2dx+\frac{1}{R_n^2}\int_{B_{R_n}(x_0)}w^2dx\right)\nonumber\\
&\leq &cR_n^{2-\ep}\|\mathbf{H}\|_{2p, B_{R}(x_0)}^2\left(\int_{B_{R_n}(x_0)}|\nabla w|^2dx+\frac{1}{R_n^2}\int_{B_{R_n}(x_0)}w^2dx\right)\nonumber\\
&= &\alpha\int_{B_{R_n}(x_0)}|\nabla w|^2dx+\frac{c^{\frac{2}{2-\ep}}\|\mathbf{H}\|_{2p, B_{R}(x_0)}^{\frac{4}{2-\ep}}}{\alpha^{\frac{2}{2-\ep}}}\int_{B_{R_n}(x_0)}w^2dx,\label{r14}
\end{eqnarray}
where $p>1, \ep\in(0,1)$ are given as in \eqref{kcl} and $\alpha>0$.
%Set 
%\begin{equation*}
%cR_n^{2-\ep}\|\mathbf{H}\|_{2p, B_{R}(x_0)}^2=\alpha.
%\end{equation*}
We can conclude from \eqref{r14}, \eqref{r15}, and \eqref{r17} that
\begin{eqnarray}
\lefteqn{
	\int_{B_{R_n}(x_0)}\frac{1}{\psi}|\mathbf{H}|^2\theta_{n+1}^2\left[(\psi-K_{n+1})^+\right]^2dx}\nonumber\\
&\leq&\alpha\int_{S_{n+1}(t)}\theta_{n+1}^2\psi^{-1}\left|\nabla\psi\right|^2dx+\frac{c4^n}{R^2}\int_{S_{n+1}(t)}\psi dx+\frac{c\|\mathbf{H}\|_{2p, B_{R}(x_0)}^{\frac{4}{2-\ep}}}{\alpha^{\frac{2}{2-\ep}}}\int_{S_{n+1}(t)}\theta_{n+1}^2\psi dx.
\end{eqnarray}
Similarly,
\begin{eqnarray}
\int_{S_{n+1}(t)} \psi|\mathbf{F}|^2\theta_{n+1}^2dx
&=&\int_{S_{n+1}(t)}\left[\left(\sqrt{\psi}-\sqrt{K_{n+1}}\right)^++\sqrt{K_{n+1}}\right]^2|\mathbf{F}|^2\theta_{n+1}^2dx\nonumber\\
&\leq&2\int_{B_{R_n}(x_0)}\left[\left(\sqrt{\psi}-\sqrt{K_{n+1}}\right)^+\right]^2|\mathbf{F}|^2\theta_{n+1}^2dx\nonumber\\
&&+2K_{n+1}\int_{S_{n+1}(t)}|\mathbf{F}|^2\theta_{n+1}^2dx\nonumber\\
&\leq&\alpha\int_{S_{n+1}(t)}\theta_{n+1}^2\psi^{-1}\left|\nabla\psi\right|^2dx+\frac{c\alpha4^n}{R^2}\int_{S_{n+1}(t)}\psi dx\nonumber\\
&&+\frac{c\|\mathbf{F}\|_{2p, B_{R}(x_0)}^{\frac{4}{2-\ep}}}{\alpha^{\frac{2}{2-\ep}}}\int_{S_{n+1}(t)}\theta_{n+1}^2\psi dx\nonumber\\
&&+cK_{n+1}\|\mathbf{F}\|_{2p, B_{R}(x_0)}^{2}|S_{n+1}(t)|^{\frac{p-1}{p}},\\
\int_{B_{R_n}(x_0)}| h|\theta_{n+1}^2(\psi-K_{n+1})^+dx&\leq &\int_{S_{n+1}(t)} \psi| h|\theta_{n+1}^2dx\nonumber\\
&\leq&\alpha\int_{S_{n+1}(t)}\theta_{n+1}^2\psi^{-1}\left|\nabla\psi\right|^2dx+\frac{c\alpha4^n}{R^2}\int_{S_{n+1}(t)}\psi dx\nonumber\\
&&+\frac{c\|h\|_{p, B_{R}(x_0)}^{\frac{2}{2-\ep}}}{\alpha^{\frac{2}{2-\ep}}}\int_{S_{n+1}(t)}\theta_{n+1}^2 \psi dx\nonumber\\
&&+cK_{n+1}\|h\|_{p, B_{R}(x_0)}|S_{n+1}(t)|^{\frac{p-1}{p}}.
\end{eqnarray}
Use the preceding three estimates in \eqref{wine12}, select $\alpha$ suitably small in the resulting inequality,  and thereby obtain
\begin{eqnarray}
\lefteqn{\io\left[\left(\sqrt{\psi}-\sqrt{K_{n+1}}\right)^+\right]^2\theta_{n+1}^2dx+ \iot\frac{1}{\psi}|\nabla (\psi-K_{n+1})^+|^2 \theta_{n+1}^2dxdt}\nonumber\\
&\leq&\frac{c4^n}{R^2}\int_{Q_{n+1}}\psi dxdt+\frac{c4^n\|(a|\q|+m)\|_{\infty, Q_0}}{R^2}\int_{Q_{R_n}(z_0)}\frac{1}{\psi} \left[(\psi-K_{n+1})^+\right]^2dxdt\nonumber\\
&&+c\sup_{0\leq t\leq T}\left(\|\mathbf{H}\|_{2p, B_{R}(x_0)}^{\frac{4}{2-\ep}}+\|\mathbf{F}\|_{2p, B_{R}(x_0)}^{\frac{4}{2-\ep}}+\|h\|_{p, B_{R}(x_0)}^{\frac{2}{2-\ep}}\right)\int_{Q_{n+1}}\psi dxdt\nonumber\\
&&+cK_{n+1}\sup_{0\leq t\leq T}\left(\|h\|_{p, B_{R}(x_0)}+\|\mathbf{F}\|_{2p, B_{R}(x_0)}^{2}\right)\int_{t_0-R_n^2}^{t_0}|S_{n+1}(t)|^{\frac{p-1}{p}}dt\nonumber\\
&\leq&\frac{c(4^n+R^2)\Gamma}{R^2}\int_{Q_{n+1}}\psi dxdt+\frac{c4^n\Gamma}{R^2}\int_{Q_{R_n}(z_0)}\frac{1}{\psi} \left[(\psi-K_{n+1})^+\right]^2dxdt\nonumber\\
&&+cK_{n+1}\Gamma|Q_{n+1}|^{\frac{p-1}{p}}.\label{rwine12}
\end{eqnarray}
where
\begin{equation}\label{gmdef}
\Gamma=\max\left\{1,\|\q\|_{\infty, Q_1}, \sup_{t_0-R^2\leq t\leq t_0}\left(\|\mathbf{H}\|_{2p, B_{R}(x_0)}^{\frac{4}{2-\ep}}+\|\mathbf{F}\|_{2p, B_{R}(x_0)}^{\frac{4}{2-\ep}}+\|h\|_{p, B_{R}(x_0)}^{\frac{2}{2-\ep}}\right) \right\}.
\end{equation}
%The last step is due to \eqref{kbt}.
%where $(1+\sigma)r$ is given as in \eqref{ldef}.
Observe that
\begin{eqnarray*}
\frac{1}{\psi} |\nabla (\psi-K_{n+1})^+|^2&=&	\frac{1}{\psi} |\nabla \psi|^2\chi_{S_{n+1}(t)}=4|\nabla\sqrt{\psi}|^2\chi_{S_{n+1}(t)}
=4|\nabla(\sqrt{\psi}-\sqrt{K_{n+1}})^+|^2,\\
\frac{1}{\psi} \left[(\psi-K_{n+1})^+\right]^2&=&\frac{1}{\psi}\left[\left(\sqrt{\psi}-\sqrt{K_{n+1}}\right)^+\right]^2\left(\sqrt{\psi}+\sqrt{K_{n+1}}\right)^2\nonumber\\
&=&\left[\left(\sqrt{\psi}-\sqrt{K_{n+1}}\right)^+\right]^2\left(1+\frac{\sqrt{K_{n+1}}}{\sqrt{\psi}}\right)^2
\leq 4\left[\left(\sqrt{\psi}-\sqrt{K_{n+1}}\right)^+\right]^2.
\end{eqnarray*}
We also have
\begin{eqnarray*}
\frac{\sqrt{K_{n+1}}-\sqrt{K_{n}}}{\sqrt{K_{n+1}}}&=&\frac{\sqrt{1-\frac{1}{2^{n+2}}}-\sqrt{1-\frac{1}{2^{n+1}}}}{\sqrt{1-\frac{1}{2^{n+2}}}}\nonumber\\
&=&\frac{1}{2^{n+2}\left(\sqrt{1-\frac{1}{2^{n+2}}}+\sqrt{1-\frac{1}{2^{n+1}}}\right)\sqrt{1-\frac{1}{2^{n+2}}}}
\geq \frac{1}{2^{n+3}}.
\end{eqnarray*}
With this in mind, we estimate
\begin{eqnarray}
\left[\left(\sqrt{\psi}-\sqrt{K_{n}}\right)^+\right]^2&\geq& \psi\left[\left(1-\frac{\sqrt{K_{n}}}{\sqrt{\psi}}\right)^+\right]^2\chi_{S_{n+1}(t)}\nonumber\\
&\geq &\psi\left(1-\frac{\sqrt{K_{n}}}{\sqrt{\sqrt{K_{n+1}}}}\right)^2\chi_{S_{n+1}(t)}\geq\frac{1}{2^{2(n+3)}}\psi\chi_{S_{n+1}(t)}.\label{r20}
\end{eqnarray}
Plugging the preceding results into \eqref{rwine12}, we obtain
\begin{equation}\label{wine14}
J_n+\ioT\left|\nabla\left(\sqrt{\psi}-\sqrt{K_{n+1}}\right)^+\right|^2\theta_{n+1}^2dxdt
\leq\frac{c8^n(1+R^2)\Gamma}{R^2}y_n+cK_{n+1}\Gamma|Q_{n+1}|^{\frac{p-1}{p}},
\end{equation}
where
%where $(1+\sigma)r$ is given as in \eqref{ldef}, 
\begin{eqnarray}
J_n&=&\sup_{t_0-R_n^2\leq t\leq t_0}\io\left|\left(\sqrt{\psi}-\sqrt{K_{n+1}}\right)^+\right|^2\theta_{n+1}^2dx,\ \ \mbox{and}\\
y_n&=&\int_{Q_{R_n}(z_0)}\left[\left(\sqrt{\psi}-\sqrt{K_{n}}\right)^+\right]^{2}dxdt
.
\end{eqnarray}
Fix a number $r\in (1,2)$.
We derive, with the aid of Poincar\'{e}'s inequality and \eqref{wine14}, that
\begin{eqnarray}
\lefteqn{\int_{t_0-R_n^2}^{t_0}\io\left((\sqrt{\psi}-\sqrt{K_{n+1}})^+\theta_{n+1}\right)^{2r}dxdt}\nonumber\\
&\leq &\int_{t_0-R_n^2}^{t_0}\left(\io\left((\sqrt{\psi}-\sqrt{K_{n+1}})^+\theta_{n+1}\right)^{2}dx\right)^{\frac{r}{2}}\left(\io\left((\sqrt{\psi}-\sqrt{K_{n+1}})^+\theta_{n+1}\right)^{\frac{2r}{2-r}}dx\right)^{\frac{2-r}{2}}dt\nonumber\\
&\leq &J_n^{\frac{r}{2}}\int_{t_0-R_n^2}^{t_0}\left(\io\left((\sqrt{\psi}-\sqrt{K_{n+1}})^+\theta_{n+1}\right)^{\frac{2r}{2-r}}dx\right)^{\frac{2-r}{2}}dt \nonumber\\
&\leq &cJ_n^{\frac{r}{2}}\int_{Q_{R_n}(z_0)}\left|\nabla\left((\sqrt{\psi}-\sqrt{K_{n+1}})^+\theta_{n+1}\right)\right|^{r}dxdt \nonumber\\
&\leq &cJ_n^{\frac{r}{2}}\left(\int_{Q_{R_n}(z_0)}\left|\nabla\left((\sqrt{\psi}-\sqrt{K_{n+1}})^+\theta_{n+1}\right)\right|^{2}dxdt\right)^{\frac{r}{2}} |Q_{n+1}|^{1-\frac{r}{2}}\nonumber\\
&\leq &cJ_n^{\frac{r}{2}}\left(\int_{Q_{R_n}(z_0)}\left|\nabla(\sqrt{\psi}-\sqrt{K_{n+1}})^+\right|^{2}\theta_{n+1}^2dxdt \right)^{\frac{r}{2}}|Q_{n+1}|^{1-\frac{r}{2}}\nonumber\\
&&+cJ_n^{\frac{r}{2}}\left(\frac{c4^{n}}{R^2}\int_{Q_{R_n}(z_0)}\left[\left(\sqrt{\psi}-\sqrt{K_{n+1}}\right)^+\right]^{2}dxdt\right)^{\frac{r}{2}}|Q_{n+1}|^{1-\frac{r}{2}}\nonumber\\
&\leq &\left(\frac{c8^n(1+R^2)\Gamma}{R^2}\right)^{r} y_n^{r}|Q_{n+1}|^{1-\frac{r}{2}}+cK_{n+1}^{^r}\Gamma^r|Q_{n+1}|^{\frac{(p-1)r}{p}+1-\frac{r}{2}}\nonumber\\
&&+
\left(\frac{c8^n(1+R^2)}{R^2}\right)^{r}\Gamma^{\frac{r}{2}}y_n^{r}|Q_{n+1}|^{1-\frac{r}{2}}+\left(\frac{c4^{n}}{R^2}\right)^{\frac{r}{2}}\Gamma^{\frac{r}{2}}K_{n+1}^{\frac{r}{2}}y_n^{\frac{r}{2}}|Q_{n+1}|^{\frac{(p-1)r}{2p}+1-\frac{r}{2}}.
%\nonumber\\
%&&\frac{c8^n(1+R^2)}{R^2}y_n+cK_{n+1}^{1+\sigma}|Q_{n+1}|^{\frac{p-1}{p}}
%&=&\frac{c4^n}{R^2}\left(\Gamma+R^{\frac{2(r-1)}{r}}\right)y_n |Q_{n+1}|^{\frac{1}{2}}.
\label{wine16}
\end{eqnarray}
Therefore,
\begin{eqnarray}
y_{n+1}&\leq&\int_{t_0-R_n^2}^{t_0}\io\left((\sqrt{\psi}-\sqrt{K_{n+1}})^+\theta_{n+1}\right)^{2}dxdt\nonumber\\
&\leq&\left(\int_{t_0-R_n^2}^{t_0}\io\left((\sqrt{\psi}-\sqrt{K_{n+1}})^+\theta_{n+1}\right)^{2r}dxdt\right)^{\frac{1}{r}}|Q_{n+1}|^{1-\frac{1}{r}}\nonumber\\
&\leq&\frac{c8^n(1+R^2)\Gamma}{R^2}y_n|Q_{n+1}|^{\frac{1}{2}}+cK\Gamma|Q_{n+1}|^{\frac{p-1}{p}+\frac{1}{2}}\nonumber\\
&&+\left(\frac{c4^{n}}{R^2}\right)^{\frac{1}{2}}\Gamma^{\frac{1}{2}}K^{\frac{1}{2}}y_n^{\frac{1}{2}}|Q_{n+1}|^{\frac{p-1}{2p}+\frac{1}{2}}.\label{r23}
\end{eqnarray}
%Observe from \eqref{r20} that
%\begin{eqnarray}
%K_{n+1}^{^2}|Q_{n+1}|^{\frac{(p-1)}{p}+\frac{1-\sigma}{2}}&=&\left(\int_{Q_{n+1}}K_{n+1}^{1+\sigma}dxdt\right)^{1+\sigma}|Q_{n+1}|^{\frac{(p-1)}{p}-\frac{1+\sigma}{2}}\nonumber\\
%&\leq &c2^{(n+3)}y_n^{1+\sigma}|Q_{n+1}|^{\frac{(p-2)}{2p}}.
%\end{eqnarray}
We easily see that
\begin{eqnarray}
y_n\geq \int_{Q_{n+1}}\left(\sqrt{K_{n+1}}-\sqrt{K_{n}}\right)^{2}dxdt\geq\frac{K}{2^{2(n+3)}}|Q_{n+1}|.
\end{eqnarray}
Consequently,
\begin{equation}
K|Q_{n+1}|^{\frac{(p-1)}{p}+\frac{1}{2}} \leq 
c2^{2(n+3)}y_n|Q_{n+1}|^{\frac{p-1}{p}-\frac{1}{2}}
\end{equation}
Use this in \eqref{r23} to obtain
\begin{eqnarray}
y_{n+1}&\leq& \frac{c8^n(1+R^2)\Gamma}{R^2}y_n\left(|Q_{n+1}|^{\frac{1}{2}}+|Q_{n+1}|^{\frac{p-1}{p}-\frac{1}{2}}+|Q_{n+1}|^{\frac{p-1}{2p}}\right)\nonumber\\
&\leq&\frac{cd^n(1+R^2)\Gamma}{R^2}y_n|Q_{n+1}|^{\frac{p-2}{2p}}
\leq \frac{cd^n(1+R^2)\Gamma}{R^2K^{\frac{p-2}{2p}}}y_n^{1+\frac{p-2}{2p}},
\end{eqnarray}
where $d>1$.
To continue, we must further assume
\begin{equation}\label{conp}
p>2.
\end{equation}In view of Lemma \ref{ynb}, if we choose $K$ so large that
\begin{equation*}
y_0\leq cK\left(\frac{R^2}{(1+R^2)\Gamma}\right)^{\frac{2p}{p-2}}
\end{equation*} then we have
\begin{equation*}
\sup_{Q_{\frac{R}{2}}(z_0)}\psi\leq K.
\end{equation*} 
On account of \eqref{conk}, it is enough for us to take
\begin{equation}
K=c\left(\frac{1+R^2}{R^2}\right)^{\frac{2p}{p-2}}y_0\Gamma^{\frac{2p}{p-2}}+2.
% \frac{c}{R^{\frac{4}{2-\ell}}}y_0+2+\|\q\|_{\frac{r}{r-1}+ Q_0}^{\frac{1}{\sigma}}+\|\mathbf{H}\|_{\frac{r}{r-1}, Q_0}^{\frac{2}{\sigma}}+\|h\|_{\frac{r}{r-1}, Q_0}^{\frac{1}{\sigma}}+ \|\mathbf{F}\|_{\frac{r}{r-1}, Q_0}^{\frac{2}{\sigma}}.
\label{wine19}
\end{equation}

Now we proceed to estimate $ \mathbf{H}, h, \mathbf{F}$.
%The boundedness of $\q$ in $L^{\infty}(Q_0)$ is a direct consequence of \eqref{qlpb}. 
%As for the remaining functions, 
We first observe from \eqref{ddf} that
\begin{equation*}
|D|\leq b|\q|+m.
\end{equation*}
Subsequently,
\begin{eqnarray*}
|D_i|&\leq & (c_1|\q|^2+c_2)|\nabla u|,\ \ i=1,2,\\
|D_3|&\leq &(c_1|\q|^2+c_2)|\nabla u|^2.
\end{eqnarray*}
It follows from \eqref{ellip} and \eqref{dpx} that
\begin{equation*}
|\mathbf{G}|\leq c|\nabla^2v|.
\end{equation*}
Note that
\begin{eqnarray*}
\mdet(D)=(a|\q|+m)(b|\q|+m).
\end{eqnarray*}
Thus we have
\begin{equation*}
|\nabla\mdet(D)|\leq c(|\q|+1)|\nabla^2v|.
\end{equation*}
We are ready to estimate
\begin{eqnarray*}
|\mathbf{F}|&=&\left|-D\mathbf{G}+2a^{-1}(\mdiv D\nabla u-\nabla u\cdot\q)D\nabla u\right|\nonumber\\
&\leq &c(|\q|+1)|\nabla^2 v|+c(|\q|^2+|\q|).
%,\\
%|\mathbf{G}|=
\end{eqnarray*}
Note that
\begin{equation*}
\psi\geq 1\ \ \mbox{if and only if}\ \ \vp\geq 1.
\end{equation*}
By \eqref{ellip}, we have
\begin{equation*}
|\nabla u |^2\geq \frac{1}{b|\q|+m}\ \ \mbox{in $Q_1$,}
\end{equation*}
from whence follows
\begin{equation*}
|\vp^{-1}u_tD\nabla u|\leq \frac{c|u_t|(b|\q|+m)}{|\nabla u|}\leq c|u_t|(b|\q|+m)^{\frac{3}{2}}\ \  \mbox{in $Q_1$.}
\end{equation*}
We are in a position to estimate
\begin{eqnarray*}
\left|\mathbf{H}\right|	&=&\left|D\mathbf{G}+2\vp^{-1}u_tD\nabla u+\frac{1}{\textup{det}(D)\vp}\left(
D_1^T\nabla\mdet(D),
D_2^T\nabla\mdet(D)
\right)\nabla u\right|\nonumber\\
&\leq& c(|\q|+1)|\nabla^2 v|+c|u_t|(b|\q|+m)^{\frac{3}{2}}\ \  \mbox{in $Q_1$.}
\end{eqnarray*}
As for $h$, we first note that all our previous calculations are still valid if we drop the two non-positive terms in $h$. Keeping this in mind, we estimate
\begin{eqnarray*}
|h|&\leq&\left|-\frac{1}{\textup{det}(D)\vp}\nabla\mdet(D)\cdot\left(D_1\mathbf{G}, D_2\mathbf{G}\right)\nabla u+\frac{2(u_t-\mdiv D\nabla u+\nabla u\cdot\q)}{\textup{det}(D)\vp^2}\nabla\mdet(D)\cdot D_3\nabla u\right|\nonumber\\
&&+\left|-2\vp^{-1}u_t\nabla u\q+\vp^{-1}D_t\nabla u\cdot\nabla u-2\vp^{-1}(u_t-\mdiv D\nabla u+\nabla u\cdot\q)D\nabla u\cdot \mathbf{G}\right|\nonumber\\
&\leq &c(|\q|+1)|\nabla^2 v|^2+c|u_t|(b|\q|+m)^{\frac{3}{2}}|\nabla^2v|+c|\q|(|\q|+1)|\nabla^2 v|^2+c|u_t|(b|\q|+m)^{\frac{3}{2}}\nonumber\\
&&+c|\nabla v_t|+c|\q|(|\q|+1)|\nabla^2 v|.
\end{eqnarray*}
With the aid of \eqref{qlpb} , for $t\in (t_0-R^2, t_0)$ we estimate that
\begin{eqnarray}
\left(\int_{S_1(t)}|\mathbf{H}|^{2p}dx\right)^{\frac{1}{p}}&\leq &c\left(\int_{S_1(t)}\left[(|\q|+1)|\nabla^2 v|\right]^{2p}dx\right)^{\frac{1}{p}}\nonumber\\
&&+c\left(\int_{S_1(t)}\left[|u_t|(b|\q|+m)^{\frac{3}{2}}\right]^{2p}dx\right)^{\frac{1}{p}}\nonumber\\
&\leq &c\|b|\q|+m\|_{\infty, Q_1}^2\|\nabla^2v\|_{2p,S_1(t)}^2+c\|u_t\|_{2p,S_1(t)}^2\|b|\q|+m\|_{\infty, Q_1}^3.
\end{eqnarray}
Similarly, 
\begin{eqnarray*}
\|\mathbf{F}\|_{2p,S_1(t)}&\leq &c\|b|\q|+m\|_{\infty, Q_1}\|\nabla^2v\|_{2p,S_1(t)}^2+c,\\
\|h\|_{p,S_1(t)}&\leq &c\|b|\q|+m\|_{\infty, Q_1}^3\|\nabla^2v\|_{2p,S_1(t)}^2+c\|u_t\|_{2p,S_1(t)}^2+c\|\nabla v_t\|_{p,S_1(t)}+c.
\end{eqnarray*}
Set
\begin{equation}\label{godef}
\Gamma_{1, Q_{R}(z_0)}=\max\left\{1, \|\q\|_{\infty, Q_{R}(z_0)}, \sup_{t_0-R^2\leq t\leq t_0}\left(\|\nabla^2v\|_{2p,B_{R}(x_0)}^2+\|u_t\|_{2p,B_{R}(x_0)}^2+\|\nabla v_t\|_{p,B_{R}(x_0)}\right)\right\}.
\end{equation}
Obviously, we can find a positive number $\ell$ large enough so that
\begin{equation*}
\Gamma^{\frac{2p}{p-2}}\leq G_{1,Q_{R}(z_0)}^\ell.
\end{equation*}
Recall that
\begin{equation*}
y_0=\int_{Q_{ R}(z_0)}\left[\left(\sqrt{\psi}-\sqrt{\frac{K}{2}}\right)^+\right]^{2}dxdt
\leq \|\vp\|_{j, Q_{ R}(z_0)}^j.
\end{equation*}
Collecting the preceding estimates in \eqref{wine19} and taking the $j^{\mbox{th}}$ root of the resulting inequality, we arrive at
\begin{equation}\label{wine20}
\sup_{Q_{\frac{R}{2}}(z_0)} \vp\leq c \left(\frac{1+R^2}{R^2}\right)^{\frac{2p}{(p-2)j}}\|\vp\|_{j, Q_{ R}(z_0)}G_{1,Q_{R}(z_0)}^{\frac{\ell}{j}}+c.
\end{equation}
 We claim that we can extend the above estimate to the whole $\ot$. That is, we have
\begin{equation}\label{wine21}
\sup_{\ot} \vp\leq c \|\vp\|_{j, Q_{ R}(z_0)}G_{1,\ot}^{\frac{\ell}{j}}+c.
\end{equation}
To see this, we appeal to the classical technique of turning a boundary point into an interior point (\cite{GT}, p. 303). We shall give a brief outline here. Let the boundary curve $\po$ be covered by a finite number of overlapping arcs, each of which can be straightened into a segment
of $\eta_1=0$ by a suitable smooth diffeomorphism $(x_1,x_2)\rightarrow (\eta_1,\eta_2)$ defined in a neighborhood of the arc. Denote by $(\tilde{u}, \tilde{v})$ our solution $(u,v)$ in the new variables. Obviously, we have $ \tilde{u}_{\eta_1}(0,\eta_2,t)=0, \ \tilde{v}(0,\eta_2,t)=0$. 
%Our solution $u$ inThe function $\tilde{\vp}$ in the new variables satisfies the boundary condition
%\begin{equation*}
%\tilde{vp}(0,\xi_1)=\tilde{u}_{\xi_1}(0,\xi_2)=0.
%\end{equation*}
We can extend $(\tilde{u}, \tilde{v})$ across the line $\eta_1=0$ by setting
\begin{equation}\label{ext}
\tilde{u}(-\eta_1,\eta_2,t)=	\tilde{u}(\eta_1,\eta_2,t),\ \ \tilde{v}(-\eta_1,\eta_2,t)=	-\tilde{v}(\eta_1,\eta_2,t).
\end{equation} 
Then the relevant boundary points become interior points. In the Appendix below (also see \cite{SW,X2}), we derive the equation  satisfied by $\tilde{u} $, i.e., equation \eqref{eut}.
%\begin{eqnarray}
%\Delta v&=& u_{x_1}\ \ \mbox{in $\ot\equiv\Omega\times(0,T)$},\\
%u_t-\mdiv\left(\left((a|\mathbf{q}|+m)I+(b-a)\frac{\q\otimes\q}{|\q|}\right)\nabla u\right)&=& -\nabla u\cdot\q\ \ \ \mbox{in $\ot$},\\
%\end{eqnarray}
Obviously, this equation
has the same basic structure as the original \eqref{gwe2}. 
%Since $\po$ is $C^{1,1}=W^{2,\infty}$, 
We can easily check from \eqref{bc} below that 
%the difference between the two equations does not have any non-trivial impact on 
the proof of \eqref{wine20} is still applicable here. %Thus 
%is that the two elliptic coefficient matrices in the new system are similar to the ones, respectively, while we have new smooth coefficients in other terms.
% contains new terms introduced by the change of variables. However,   they appear roughly as smooth ``multiples'' of the original coefficients, and thus the previous calculations still carry through. We shall omit the details.
% do not result in any new difficulty as far as we are concerned.
%introduces someis of the same type as the one satisfied by $(u,v)$. We refer the reader to   for more details. 
%W
Therefore, we can conclude that
  \eqref{wine20} is still valid for $z_0\in \Sigma_T$. If $t_0=0$, then we just need to change $Q_{R_n}(z_0)$ to $B_{R_n}\times[0, R_n^2)$ and require 
\begin{equation*}
K>\max\{2\max_{\Omega}\psi(x,0),2\}
\end{equation*}
instead of \eqref{conk} in the proof of \eqref{wine20}. Subsequently, \eqref{wine21} follows.
%By an argument in ,
%To be more specific, we employ the classical technique of 
%\nonumber\\

%Without loss of generality, assume $\|\vp\|_{\infty,\ot}\geq 1$.
 We deduce from the classical Calder\'{o}n-Zygmund inequality, the Sobolev embedding theorem,  \eqref{pplr}, and Lemma \ref{plin} that
\begin{eqnarray}
\sup_{\ot} \vp &\leq & c\|\vp\|_{j, \ot}\left(\sup_{0\leq t\leq T}\left(\|\nabla u\|_{2p,\Omega}^2+\|u_t\|_{2p,\Omega}^2+\|u_t\|_{p,\Omega}\right)+1\right)^{\frac{\ell}{j}}+c\nonumber\\
&\leq & c\|\vp\|_{j, \ot}\left(\|\nabla u\|_{\infty,\ot}^{4+2q}+1\right)^{\frac{\ell}{j}}+c\nonumber\\
&\leq & c\|\vp\|_{j, \ot}\|\vp\|_{\infty,\ot}^{\frac{(q+2)\ell}{j}}+c\|\vp\|_{j, \ot}+c,
\end{eqnarray}
where $q>2$ is given as in Lemma \ref{plin}. On account of Lemma \ref{ue} and the interpolation inequality (\cite{GT}, p. 146), $\|\vp\|_{j, \ot}\leq \ep\|\vp\|_{\infty, \ot}+\frac{1}{\ep^{j-1}} \|\vp\|_{1, \ot}\leq \ep\|\vp\|_{\infty, \ot}+\frac{c}{\ep^{j-1}}$. It follows
\begin{equation}\label{r11}
\sup_{\ot} \vp\leq c\|\vp\|_{j, \ot}\|\vp\|_{\infty,\ot}^{\frac{(q+2)\ell}{j}}+c\leq  cT^{\frac{1}{j}}\|\vp\|_{\infty,\ot}^{1+\frac{(q+2)\ell}{j}}+c.
\end{equation} Evidently,  \eqref{r11} implies that $\vp\in L^\infty(\ot) $ whenever $\vp\in L^j(\ot)$ with $\frac{(q+2)\ell}{j}<1$. In view of \eqref{r22}, this is true if $\q$ satisfies \eqref{r30}. Unfortunately, we do not have \eqref{r30} at this point. So we will first show $\vp\in L^\infty(\ot) $ for $T$ suitably small. Then extend the solution in the time direction. To this end, we let
\begin{equation}\label{cont}
T\leq 1.
\end{equation}
Then we may assume that $c$ in \eqref{r11} is independent of $T$. To see this, we check our earlier proof and conclude
%But it is not difficult to see that we may take may also depend on
$c=c_1T^{\delta_0}+c_2$ for some $\delta_0\geq 0$. Of course, here 
 $\delta_0, c_1, c_2$ are independent of $T$. There is
%To see this, we simply 
another way to reach this conclusion. That is, consider the function $u_T(x,t)=u(x, Tt)$ on $\Omega_1$. We shall omit the details.

Let \eqref{cont} hold and set \begin{equation*}
\ep=cT^{\frac{1}{j}},\ \ \ \delta=\frac{(q+2)\ell}{j}.
\end{equation*}
Consider the function $g(s)=\ep s^{1+\delta}-s+c$ on $(0, \infty)$. Then \eqref{r11} is equivalent to $g(\|\vp\|_{\infty,\ot})\geq 0$. The function $g$ achieves its minimum value at $s_0=\frac{1}{\left[\ep(1+\delta)\right]^{\frac{1}{\delta}}}$. The minimum value
\begin{eqnarray*}
g(s_0)&=&\frac{\ep}{\left[\ep(1+\delta)\right]^{\frac{1+\delta}{\delta}}}-\frac{1}{\left[\ep(1+\delta)\right]^{\frac{1}{\delta}}}+c\nonumber\\
&=&c-\frac{\ep\delta}{\left[\ep(1+\delta)\right]^{\frac{1+\delta}{\delta}}}\leq -\ep,
\end{eqnarray*}
provided that $(c+\ep)\ep^{\frac{1}{\delta}} \leq \frac{\delta}{(1+\delta)^{\frac{1+\delta}{\delta}}}$. We further require that $\|\vp(\cdot,0)\|_{\infty,\Omega}\leq s_0$. Denote by $T_0$ the largest $T$ that satisfies the preceding two conditions and \eqref{cont}. Then we have
% and \eqref{cont} on $T$ imply that there is a positive number $c_0$ determined by the given data such that 
\begin{equation}\label{rr1}
\|\vp\|_{\infty,\Omega\times[0,T_0]}\leq s_0.
\end{equation}
 By the discussion at the end of this section, we can deduce from \eqref{rr1} that $u\in L^\infty(0,T; C^2(\overline{\Omega}))$. In view of (H2), this is enough for us to be able to extend the solution in the time direction as far away as we want (see \cite{X7} for more details).

%\section{existence}
So far we have assumed that there is a weak solution $(u, v)$ to \eqref{gwe1}-\eqref{gwe5} which is so regular that all the preceding calculations are legitimate. We can do so if we can show that the weak solution can be viewed as the limit of a sequence of smooth approximate solutions $(u_\varepsilon, v_\varepsilon), \varepsilon\in(0.1)$, and the preceding estimates for $(u, v)$ are satisfied by $(u_\varepsilon, v_\varepsilon)$ uniformly in $\varepsilon$. In the remaining part of the proof, we will show how we can achieve this. The idea is to design a suitable approximation scheme for \eqref{gwe1}-\eqref{gwe5}. The key is to find a way to approximate the term $|\q|$ without destroying the basic structure of the original system.
% so that all the calculations in the preceding two sections are valid. , while maintaining as much as possible
To do this, let $\zeta$ be a mollifier on $\mathbb{R}^{3} $. That is, $\zeta$ is a compactly supported $C^\infty$ function with the properties
\begin{eqnarray*}
\int_{\mathbb{R}^{3}}\zeta dz=\int_{\mathbb{R}^{3}}\zeta dxdt&=&1,\\
\zeta^{(\varepsilon)}(z)=\frac{1}{\varepsilon^3}\zeta\left(\frac{z}{\varepsilon}\right)&\rightarrow& \delta (z)\ \ \mbox{in the sense of distributions as $\varepsilon\rightarrow 0^+$,}
\end{eqnarray*}
where 
$\delta( z )$
is the Dirac delta function.
%In our situation a function $v$ in a function space on $\ot$ always has a continuous extension in the same class of functions on $\mathbb{R} ^3$. For example, if Define $v$ to be zero outside $\ot$. 
Then set
\begin{eqnarray*}
v^{(\varepsilon)}&=&\zeta^{(\varepsilon)}*v,\\
\qe&=&\left(\begin{array}{c}
-(v^{(\varepsilon)})_{x_2}\\
(v^{(\varepsilon)})_{x_1}
\end{array}\right).
\end{eqnarray*}
Outside  $\ot$ the function $v$ is understood to be the extension of $v$  in the sense of the Sobolev extension theorem (\cite{EG}, p.135).
Let
\begin{equation*}
D_\varepsilon=\left(a(|\qe|^2+\varepsilon)^{\frac{1}{2}} +m\right)I+\frac{b-a}{( |\qe|^2+\varepsilon)^{\frac{1}{2}}}\qe\otimes\qe.
\end{equation*}
Obviously, $D_\varepsilon$ is infinitely differentiable for each $\varepsilon>0$. Moreover,
\begin{equation*}
m|\xi|^2\leq D_\varepsilon\xi\cdot\xi=\left(a(|\qe|^2+\varepsilon)^{\frac{1}{2}} +m\right)|\xi|^2 +\frac{b-a}{(|\qe|^2+\varepsilon)^{\frac{1}{2}}}(\xi\cdot\qe)^2\leq\left(b(|\qe|^2+\varepsilon)^{\frac{1}{2}} +m\right)|\xi|^2
\end{equation*}
for each $\xi\in \mathbb{R}^2$.
We form our approximate problems as follows:
\begin{eqnarray}
\Delta v&=&u_{x_1}\ \ \mbox{in $\ot$,}\label{ap1}\\
u_t-\mdiv\left(D_\varepsilon\nabla u\right)&=&-\nabla u\cdot\q\ \ \mbox{in $\ot$,}\label{ap2}\\
v=D_\varepsilon\nabla u\cdot\nu&=&0\ \ \mbox{on $\Sigma_T$,}\label{ap3}\\
u(x,0)&=& u_0(x)\ \ \mbox{on $\Omega$}.\label{ap4}
\end{eqnarray} 
As before, here $\q=(-v_{x_2}, v_{x_1})$.

The existence of a solution to \eqref{ap1}-\eqref{ap4} can be obtained via the
Leray-Schauder fixed point theorem (\cite{GT}, p.280). To see this, we define an operator $\mathbb{T}$ from $L^\infty(\ot)$ into itself as follows: We say $\mathbb{T}(w)=u$ if $u$ is the solution
of the problem
\begin{eqnarray}
u_t-\mdiv\left(D_\varepsilon\nabla u\right)&=&-\nabla u\cdot\q\ \ \mbox{in $\ot$,}\label{ap5}\\
D_\varepsilon\nabla u\cdot\nu&=&0\ \ \mbox{on $\Sigma_T$,}\label{ap6}\\
u(x,0)&=& u_0(x)\ \ \mbox{on $\Omega$}.\label{ap7}
\end{eqnarray}
The functions $\q, \qe$ in the problem are given by first solving
\begin{eqnarray}
\Delta v&=&w_{x_1}\ \ \mbox{in $\ot$,}\label{ap8}\\
v&=&0\ \ \mbox{on $\Sigma_T$}
\end{eqnarray}
and then letting 
\begin{equation*}
\q=\left(\begin{array}{c}
-v_{x_2}\\
 v_{x_1}
\end{array}\right),\ \ \ \qe=\left(\begin{array}{c}
-(v^{(\varepsilon)})_{x_2}\\
(v^{(\varepsilon)})_{x_1}
\end{array}\right).
\end{equation*}
As before, $v^{(\varepsilon)}=\zeta^{(\varepsilon)}*v$.
To see that $\mathbb{T}$ is well-defined, we infer from \eqref{qlpb}
%a result in (\cite{R}, p.82)
%the Calder\'{o}n-Zygmund inequality 
that for each $r>1$ there is a positive number $c$ such that
\begin{eqnarray}\label{vlrb}
\|\nabla v\|_{r,\Omega}\leq c\|w\|_{r,\Omega}\leq c.
\end{eqnarray}
In particular, the number $c$ in the above inequality is independent of $t$. Thus $\q$ is a function in $L^\infty(0,T; (L^{r}(\Omega))^2)$ for each $r>1$.  %It immediately follows from the Sobolev embedding theorem  that
%\begin{equation}
%\sup_{\ot}|\q|\leq c.
%\end{equation}
Equation \eqref{ap5} is linear and uniformly parabolic in $u$. Classical results assert that there is a unique weak solution $u$ to \eqref{ap5}-\eqref{ap7} in the space
$C[0, T; L^2(\Omega)]\cap L^2(0,T; W^{1,2}(\Omega))$. Furthermore, $u$ is H\"{o}lder continuous on $\overline{\ot}$. Thus we can conclude that $\mathbb{T}$ is continuous and maps bounded sets into precompact ones. It remains to be seen that we have
\begin{equation}\label{uub}
\|u\|_{\infty,\ot}\leq c
\end{equation}
for all $u\in L^\infty(\ot)$ and all $\sigma\in (0, 1)$ satisfying $u=\sigma \mathbb{T}(u)$. This equation is equivalent to the following
\begin{eqnarray*}
\Delta v&=&u_{x_1}\ \ \mbox{in $\ot$,}\label{ap11}\\
u_t-\mdiv\left(D_\varepsilon\nabla u\right)&=&-\nabla u\cdot\q\ \ \mbox{in $\ot$,}\label{ap12}\\
v=D_\varepsilon\nabla u\cdot\nu&=&0\ \ \mbox{on $\Sigma_T$,}\label{ap13}\\
u(x,0)&=& \sigma u_0(x)\ \ \mbox{on $\Omega$}.\label{ap14}
\end{eqnarray*}
We can easily infer \eqref{uub} from Lemma \ref{ump}.

We can employ a bootstrap argument to gain high regularity on the solution $(u, v)$. We begin with
\begin{equation*}
u\in C^{\alpha,\alpha/2}(\overline{\ot})\cap L^2(0,T; W^{1,2}(\Omega)) \ \ \mbox{for some $\alpha\in (0,1)$}.
\end{equation*}
This together with the Calder\'{o}n-Zygmund inequality and \eqref{vlrb} implies
\begin{eqnarray*}
v\in L^\infty(0,T;W^{1,r}(\Omega) )\cap L^2(0,T; W^{2,2}(\Omega))\ \ \mbox{for each $r>1$}.
\end{eqnarray*}
We can write \eqref{ap2} in the form
\begin{equation*}
u_t-\mdiv\left(D_\varepsilon\nabla u\right)=-\mdiv\left( u\q\right)\ \ \mbox{in $\ot$.}
\end{equation*}
Remember that entries of $D_\varepsilon$ are infinitely differentiable. 
%We are in a position to infer from a result in (\cite{BLP}, p. 273-274) that for each $r>2$ 
By \eqref{r22}, there is a positive number $c$ such that
\begin{equation*}
\|\nabla u\|_{r,\ot}\leq c\|u\q\|_{r,\ot}+c<\infty.
\end{equation*}
%Classical results in (\cite{LSU}, Chap. IV) become applicable. Upon using them
%appropriately, we can conclude that
%The parabolic version of \eqref{vlrb} asserts that
% $|\nabla u|\in L^r(\ot)$ for each $r>1$.
% is H\'{o}lder continuous on $\overline{\ot}$. 
%This combined with \eqref{ap8} implies that $\nabla v$ has the same property. That is, $D_\varepsilon$ has H\"{o}lder continuous entries. This puts up in a position to apply 
%According to the , we can conclude 
%By classical results in (\cite{LSU}, Chap. IV) that
Now classical results in (\cite{LSU}, Chap. IV) become applicable. Upon using them
appropriately, we can conclude that $u_t, \Delta u\in L^r(\ot)$ for each $r>1$, from which it follows that $\nabla v\in L^r(0,T; (W^{2,r}(\Omega))^2)$ for each $r>1$. By differentiating \eqref{ap2} with respect to $x_i, i=1,2$, 
 we arrive at $(u_{x_i})_t, \Delta u_{x_i}\in L^r(\ot)$ for each $r>1$. Furthermore, $\q, u_t, \nabla u$ are all H\"{o}lder continuous on $\overline{\ot}$. Thus we can conclude from the Schauder approach (\cite{GT}, Chp. 6) that $u\in L^\infty(0,T, C^2(\overline{\Omega}))$.
 %This combined with the fact that 
% the $L^p$ norms of $\qe$ (resp. its partial derivatives) are bounded by their corresponding norms of $\q$ (resp. its partial derivatives).
 % is sufficient for all the calculations in the preceding sections to carry through here.
 
Now denote the solution to \eqref{ap1}-\eqref{ap4} by $(u_\varepsilon,v_\varepsilon )$. 
%As we indicated earlier, 
Clearly, $(u_\varepsilon,v_\varepsilon )$ is regular enough for us to apply the proof of \eqref{rr1}, thereby obtaining
\begin{equation*}
\|\nabla u_\varepsilon\|_{\infty,\ot}\leq c.
\end{equation*}
By virtue of \eqref{ap1} and \eqref{ap3}, 
  for each $r>1$ there is a positive number $c$ such that
\begin{equation*}
\|v_\varepsilon\|_{L^\infty(0,T; W^{2,r}(\Omega))}\leq c,
\end{equation*}
from whence follows
%Now we have
\begin{equation*}
m|\xi|^2\leq D_\varepsilon\xi\cdot\xi\leq c|\xi|^2\ \ \mbox{for each $\xi\in \mathbb{R}^2$.}
\end{equation*}
%equation \eqref{ap2} is also uniformly elliptic in $\varepsilon$. 
This yields that $\{u_\varepsilon\}$ is precompact in $C(\overline{\ot})$. For $\varepsilon_1,\varepsilon_2\in (0,1)$ we can derive from \eqref{ap1} and \eqref{ap3} that
\begin{equation*}
\io|\nabla v_{\varepsilon_1}-\nabla v_{\varepsilon_2}|^2dx=\io(u_{\varepsilon_1}-u_{\varepsilon_2})(v_{\varepsilon_1}-v_{\varepsilon_2})_{x_1}dx\leq c\|u_{\varepsilon_1}-u_{\varepsilon_2}\|_{\infty,\ot}.
\end{equation*}
Subsequently, $\{v_\varepsilon\}$ is precompact in $L^\infty(0, T; W^{1,2}(\Omega))$. This implies the precompactness of $D_\varepsilon$ in  $L^\infty\left(0, T; \left(L^{2}(\Omega)\right)^{2\times 2}\right)$. Now we have enough estimates to justify
%We can easily collect enough additional estimates to enable us to 
passing to the limit in \eqref{ap1}-\eqref{ap4}. The proof is complete.
\end{proof}
\begin{center}
	
	Appendix 
\end{center}
	
	In this appendix, 
	%we describe the technical procedures involved in proof of the boundary estimate. First 
	we derive the system satisfied by $(\tilde{u},\tilde{v})$ mentioned earlier in the proof of the boundary estimate.
	To this end, we suppose that our diffeomorphism $\mathbb{B}$ is given as follows: We say
	\begin{equation}\label{app1}
\mathbb{B}\left(\begin{array}{c}
x_1\\
x_2
\end{array}\right)=\left(\begin{array}{c}
\eta_1\\
\eta_2
\end{array}\right)\ \ \mbox{if }\ \ \left(\begin{array}{c}
\eta_1\\
\eta_2
\end{array}\right)=\left(\begin{array}{c}
g_1(x_1,x_2)\\
g_2(x_1,x_2)
\end{array}\right)\equiv\mathbf{g},
	\end{equation}
	%the relevant portion of the 
	%Here we just need to assume that 
	where $g_1, g_2$ are two $ C^{2, \gamma}$ functions, where $\gamma\in (0,1)$, defined in a neighborhood $O$ of a boundary point $x_0=(x_1^0, x_2^0 )$ with the property
		\begin{equation*}
	g_1(x_1,x_2)=0 \ \ \ \mbox{for $(x_1,x_2)\in O\cap \po$.}
	\end{equation*}
	Set $\eta_2^0=g_2(x_1^0,x_2^0)$. Assume that $\mathbb{B}$ maps $O\cap\Omega$ %under $\mathbb{B}$ is 
	onto a rectangular region $B^+\equiv [0, \delta)\times(\eta_2^0-\delta,\eta_2^0+\delta )$ for some $\delta>0$.
	 Denote by $J_{\mathbf{g}}$ the Jacobian matrix of $\mathbb{B}$, i.e.,  
		\begin{equation*}
	J_{\mathbf{g}}=\nabla\left(\begin{array}{c}
	g_1\\
	g_2
	\end{array}\right)=\left(\begin{array}{ll}
(g_1)_{x_1}&(g_2)_{x_1}\\
(g_1)_{x_2}&(g_2)_{x_2}
	\end{array}\right)
	\end{equation*}
	% we can choose $g_1, g_2$ so that
In light of \cite{SW}, we may assume that
	\begin{equation}
c_1\geq	\mdet J_{\mathbf{g}}\geq c_0\ \ \mbox{on $O$ for some $c_1\geq c_0>0$}.
	\end{equation}
%	Of course, we have
%	has determinant $1$.
By the inverse function theorem, we can solve $x_1, x_2$ from the second equation in \eqref{app1} to obtain
	\begin{equation}
\left(\begin{array}{c}
x_1\\
x_2
\end{array}\right)=\mathbb{B}^{-1}\left(\begin{array}{c}
\eta_1\\
\eta_2
\end{array}\right)=\left(\begin{array}{c}
f_1(\eta_1,\eta_2)\\
f_2(\eta_1,\eta_2)
\end{array}\right)\equiv\mathbf{f}.
	\end{equation}
	Subsequently,
	\begin{eqnarray}
	g_1(f_1(\eta_1,\eta_2),f_2(\eta_1,\eta_2))&=&\eta_1,\label{app2}\\	g_2(f_1(\eta_1,\eta_2),f_2(\eta_1,\eta_2))&=&\eta_2.\label{app5}
	\end{eqnarray}
	Differentiating \eqref{app2} with respect to $\eta_1, \eta_2$, respectively, yields
	\begin{eqnarray*}
		(g_1)_{x_1}	(	f_1)_{\eta_1}+	(	g_1)_{x_2}	(	f_2)_{\eta_1}&=&1,\\
	(	g_1)_{x_1}	(	f_1)_{\eta_2}+	(	g_1)_{x_2}	(	f_2)_{\eta_2}&=&0.
		\end{eqnarray*}
%	Similarly,
%By the same token,
Do the same to \eqref{app5} to get
		\begin{eqnarray*}
	(	g_2)_{x_1}	(	f_1)_{\eta_1}+	(	g_2)_{x_2}	(	f_2)_{\eta_1}&=&0,\\
	(	g_2)_{x_1}	(	f_1)_{\eta_2}+	(	g_2)_{x_2}	(	f_2)_{\eta_2}&=&1.
	\end{eqnarray*}
	Hence,
	\begin{equation}\label{app3}
	J_{\mathbf{f}}J_{\mathbf{g}}\circ \mathbf{f}=I	.
	\end{equation}
	Remember that
	\begin{equation*}
	(u(x_1,x_2), v(x_1,x_2))=(\tilde{u}(g_1(x_1,x_2), g_2(x_1,x_2)),\tilde{v}(g_1(x_1,x_2), g_2(x_1,x_2))).
	\end{equation*} We compute
	\begin{eqnarray*}
	u_{x_1}&=&\tilde{u}_{\eta_1}(g_1)_{x_1}+\tilde{u}_{\eta_2}(g_2)_{x_1},\\
%	u_{x_1x_1}&=&(\tilde{u}_{\eta_1\eta_1}(\eta_1)_{x_1}+\tilde{u}_{\eta_1\eta_2}(\eta_2)_{x_1})(\eta_1)_{x_1}+\tilde{u}_{\eta_1}(\eta_1)_{x_1x_1}\\
%	&&+(\tilde{u}_{\eta_2\eta_1}(\eta_1)_{x_1}+\tilde{u}_{\eta_2\eta_2}(\eta_2)_{x_1})(\eta_2)_{x_1}+\tilde{u}_{\eta_2}(\eta_2)_{x_1x_1},\\
	u_{x_2}&=&\tilde{u}_{\eta_1}(g_1)_{x_2}+\tilde{u}_{\eta_2}(g_2)_{x_2}.
	%,\\\equiv\alpha_2\equiv\alpha_1
	%	u_{x_2x_2}&=&(\tilde{u}_{\eta_1\eta_1}(\eta_1)_{x_2}+\tilde{u}_{\eta_1\eta_2}(\eta_2)_{x_2})(\eta_1)_{x_2}+\tilde{u}_{\eta_1}(\eta_1)_{x_2x_2}\\
%	&&+(\tilde{u}_{\eta_2\eta_1}(\eta_1)_{x_2}+\tilde{u}_{\eta_2\eta_2}(\eta_2)_{x_2})(\eta_2)_{x_2}+\tilde{u}_{\eta_2}(\eta_2)_{x_2x_2}.
	\end{eqnarray*}
That is,
\begin{equation*}
\nabla u=J_{\mathbf{g}}\nabla \tilde{u}.
%=J_{\mathbf{f}}^{-1}\nabla \tilde{u}.
\end{equation*}
%The last step is due to \eqref{app3}.
Similarly,
\begin{equation}\label{fr1}
\nabla v=J_{\mathbf{g}}\nabla \tilde{v}\equiv\left(\begin{array}{c}
\beta_1\\
\beta_2
\end{array}\right).
\end{equation}
Recall that
\begin{eqnarray*}
\q&=&\left(\begin{array}{c}
-v_{x_2}\\
v_{x_1}
\end{array}\right)=\left(\begin{array}{c}
-\beta_2\\
\beta_1
\end{array}\right)\equiv \tilde{\mathbf{q}},\\
D&=&(a|\mathbf{q}|+m)I+(b-a)\frac{\q\otimes\q}{|\q|}\nonumber\\
&=&(a|\tilde{\mathbf{q}}|+m)I+(b-a)\frac{\tilde{\mathbf{q}}\otimes\tilde{\mathbf{q}}}{|\tilde{\mathbf{q}}|}\equiv\tilde{D}.
\end{eqnarray*}
Set
\begin{eqnarray*}
	h_1&=&(g_1)_{x_2x_1}(f_1)_{\eta_1}+(g_1)_{x_2x_2}(f_2)_{\eta_1}+(g_2)_{x_2x_1}(f_1)_{\eta_2}+(g_2)_{x_2x_2}(f_2)_{\eta_2},\\
	h_2&=&-\left((g_1)_{x_1x_1}(f_1)_{\eta_1}+(g_1)_{x_1x_2}(f_2)_{\eta_1}+(g_2)_{x_1x_1}(f_1)_{\eta_2}+(g_2)_{x_1x_2}(f_2)_{\eta_2}\right).
\end{eqnarray*}
By our earlier assumptions, we have
\begin{equation}\label{bc}
h_1,\ h_2\in L^\infty(O). 
\end{equation}
We are ready to derive the system satisfied by $(\tilde{u},\tilde{v})$. First we have
\begin{eqnarray}
\Delta v&=&(\beta_1)_{\eta_1}(g_1)_{x_1}+(\beta_1)_{\eta_2}(g_2)_{x_1}+(\beta_2)_{\eta_1}(g_1)_{x_2}+(\beta_2)_{\eta_2}(g_2)_{x_2}\nonumber\\
&=&\left(\beta_1(g_1)_{x_1}+\beta_2(g_1)_{x_2}\right)_{\eta_1}+\left(\beta_1(g_2)_{x_1}+\beta_2(g_2)_{x_2}\right)_{\eta_2}\nonumber\\
&&-\beta_1\left((g_1)_{x_1x_1}(f_1)_{\eta_1}+(g_1)_{x_1x_2}(f_2)_{\eta_1}+(g_2)_{x_1x_1}(f_1)_{\eta_2}+(g_2)_{x_1x_2}(f_2)_{\eta_2}\right)\nonumber\\
&&-\beta_2\left((g_1)_{x_2x_1}(f_1)_{\eta_1}+(g_1)_{x_2x_2}(f_2)_{\eta_1}+(g_2)_{x_2x_1}(f_1)_{\eta_2}+(g_2)_{x_2x_2}(f_2)_{\eta_2}\right)\nonumber\\
&=&\mdiv\left(J_{\mathbf{g}}^T\left(\begin{array}{c}
\beta_1\\
\beta_2
\end{array}\right)\right)+(h_1,h_2)\left(\begin{array}{c}
-\beta_2\\
\beta_1
\end{array}\right)\nonumber\\
&=&\mdiv\left(J_{\mathbf{g}}^TJ_{\mathbf{g}}\nabla \tilde{v}\right)+(h_1,h_2)\tilde{\mathbf{q}}.\label{app4}
\end{eqnarray}
%Set $\eta_2^0=g_2(x_1^0,x_2^0)$. Then
% We can conclude that \eta_2^0=g_2(x_1^0,x_2^0)\times(\eta_2^0-\delta,\eta_2^0+\delta )
Then  $\tilde{v}$ satisfies the equation
\begin{equation}\label{app6}
\mdiv\left(J_{\mathbf{g}}^TJ_{\mathbf{g}}\nabla \tilde{v}\right)+(h_1,h_2)\tilde{\mathbf{q}}=\tilde{u}_{\eta_1}(g_1)_{x_1}+\tilde{u}_{\eta_2}(g_2)_{x_1}\ \ \mbox{in $B^+$ }
\end{equation}
for  $t\in (0,T)$. Obviously, we must compose the coefficients in the above equation with $\mathbf{f}$, i.e., $J_{\mathbf{g}}^TJ_{\mathbf{g}}=J_{\mathbf{g}}^TJ_{\mathbf{g}}\circ\mathbf{f}$, and so on. The same is understood in the subsequent equation for $\tilde{u}$. It is not difficult for us to see from \eqref{fr1}, \eqref{app4}, and the change of variables formula for the Lebesgue integrals (\cite{EG}, p.99) that 
\begin{equation*}
v\in W^{2,p}(O\cap\Omega)\ \ \mbox{if and only if}\ \ \tilde{v}\in W^{2,p}(B^+)
\end{equation*}
and $\|v\|_{ W^{2,p}(O\cap\Omega)}$ and $\|\tilde{v}\|_{W^{2,p}(B^+)}$ are comparable, where $p\geq 1$. In this sense, all the global estimates such as those in \eqref{fr} we have derived for $v$ (resp. $u$) are still valid for $\tilde{v}$ (resp. $\tilde{u}$). Thus we do not need \eqref{app6} to get estimates for $\tilde{v}$.

To obtain the equation for $\tilde{u}$, we calculate
\begin{eqnarray*}
D\nabla u&= &\left[(a|\mathbf{q}|+m)I+(b-a)\frac{\q\otimes\q}{|\q|} \right]\nabla u\nonumber\\
&=&\tilde{D}J_{\mathbf{g}}\nabla \tilde{u}.
%\equiv\left(\begin{array}{c}
%\alpha_1\\
%\alpha_2
%\end{array}\right).
\end{eqnarray*}
In view of \eqref{app4}, we have
\begin{eqnarray*}
\mdiv\left(D\nabla u\right)
&=&\mdiv\left(J_{\mathbf{g}}^T\tilde{D}J_{\mathbf{g}}\nabla \tilde{u}\right)+(h_2,-h_1)\tilde{D}J_{\mathbf{g}}\nabla \tilde{u}.
\end{eqnarray*}
%\tilde{O}\equiv
Finally, we arrive at
\begin{equation}\label{eut}
\tilde{u}_t-J_{\mathbf{g}}^T\tilde{D}J_{\mathbf{g}}:\nabla^2 \tilde{u}=\mdiv\left(J_{\mathbf{g}}^T\tilde{D}J_{\mathbf{g}}\right)\nabla\tilde{u}+(h_2,-h_1)\tilde{D}J_{\mathbf{g}}\nabla \tilde{u} -J_{\mathbf{g}}\nabla \tilde{u}\cdot\tilde{\mathbf{q}}\equiv\tilde{w}\ \mbox{in $B^+\times(0,T)$.}
\end{equation}
We can extend the above equation across the line $\eta_1=0$ appropriately 
%do the even extension indicated earlier 
so that the resulting equation is satisfied in  $(-\delta, \delta)\times(\eta_2^0-\delta,\eta_2^0+\delta )\times(0,T)$. It suffices for us to do the even extension for $\tilde{u}, J_{\mathbf{g}}^T\tilde{D}J_{\mathbf{g}}, \tilde{w}$, respectively.  Remember that the even extension of a Sobolev function is still a Sobolev function. Thus the coefficients in the resulting equation are regular enough for all  our purposes.

\end{document}